\documentclass{amsart}
\usepackage{graphics,graphics,bbm,hyperref,tikz,blkarray,xcolor,mathabx,mathtools,tikz-cd,fullpage,enumitem,mathrsfs,wasysym}
\hypersetup{colorlinks,citecolor=red,filecolor=black,linkcolor=blue,urlcolor=blue}
\tikzstyle{point}=   [draw, black, fill,shape=circle, minimum size=4pt, inner sep=0pt,outer sep=0pt]
\tikzstyle{redpoint}=   [draw, very thick, red, fill=white,shape=circle, minimum size=4pt, inner sep=0pt,outer sep=0pt]
\tikzstyle{label}=   [black,shape=circle, minimum size=4pt, inner sep=0pt,outer sep=0pt]

\definecolor{light}{gray}{.75}
\definecolor{pale}{gray}{.85}
\definecolor{med}{gray}{.5}
\definecolor{dark}{gray}{.25}

\newcommand{\Blue}[1]{{\color{blue}{#1}}}

\DeclareMathOperator{\Dep}{Dep}
\newcommand{\els}{\mathsf{els}}
\DeclareMathOperator{\im}{im}

\DeclareMathOperator{\Par}{Par}

\DeclareMathOperator{\Sur}{Sur}

\newtheorem{theorem}{Theorem}[section]

\newtheorem{prop}[theorem]{Proposition}
\newtheorem{conj}[theorem]{Conjecture}
\newtheorem{cor}[theorem]{Corollary}

\usepackage{relsize}
\makeatletter
\newtheorem*{rep@theorem}{\rep@title}
\newcommand{\newreptheorem}[2]{%
\newenvironment{rep#1}[1]{%
 \def\rep@title{#2 \ref{##1}}%
 \begin{rep@theorem}}%
 {\end{rep@theorem}}}
\makeatother

\newreptheorem{theorem}{Theorem}

\theoremstyle{definition}
\newtheorem{remark}[theorem]{Remark}
\newtheorem{example}[theorem]{Example}

\newtheorem{prob}[theorem]{Problem}
\newtheorem{definition}[theorem]{Definition}

\newcommand{\0}{\emptyset}
\newcommand{\bd}{\partial} 
\newcommand{\defterm}[1]{\boldmath\textbf{#1}\unboldmath}
\newcommand{\excise}[1]{} 
\newcommand{\ov}[1]{\overline{{#1}}}
\newcommand{\sm}{\setminus}
\newcommand{\x}{\times}

\newcommand{\N}{\mathcal{N}}
\newcommand{\A}{\mathcal{A}}

\newcommand{\ep}{\varepsilon}
\newcommand{\Z}{\mathbb{Z}}
\newcommand{\R}{\mathbb{R}}
\newcommand{\HH}{\tilde H} 
\newcommand{\fancyB}{\mathscr{B}}

\makeatletter
\def\moverlay{\mathpalette\mov@rlay}
\def\mov@rlay#1#2{\leavevmode\vtop{%
   \baselineskip\z@skip \lineskiplimit-\maxdimen
   \ialign{\hfil$\m@th#1##$\hfil\cr#2\crcr}}}
\newcommand{\charfusion}[3][\mathord]{
    #1{\ifx#1\mathop\vphantom{#2}\fi
        \mathpalette\mov@rlay{#2\cr#3}
      }
    \ifx#1\mathop\expandafter\displaylimits\fi}
\makeatother

\newcommand{\cupdot}{\charfusion[\mathbin]{\cup}{\cdot}}

\newcommand{\dju}{\cupdot}

\title{A positivity phenomenon in Elser's Gaussian-cluster percolation model}
\thanks{This work was completed in part at the 2018 Graduate Research Workshop in Combinatorics, which was supported in part by NSF grants \#1604458 and \#1603823, NSA grant \#H98230-18-1-0017,  a generous award from the Combinatorics Foundation, and Simons Foundation Collaboration Grants \#426971 (to M. Ferrara) and \#316262 (to S. Hartke).\\ 
\indent JLM was supported in part by a grant from the Simons Foundation (Grant Number 315347).}

\author[GDB]{Galen Dorpalen-Barry}
\address{School of Mathematics, University of Minnesota, United States}
\email{\href{mailto:dorpa003@umn.edu}{dorpa003@umn.edu}}

\author[CH]{Cyrus Hettle}
\address{School of Mathematics, Georgia Institute of Technology, United States}
\email{\href{mailto:chettle@gatech.edu}{chettle@gatech.edu}}

\author[DCL]{David C.\ Livingston}
\address{Department of Mathematics, University of Wyoming, United States}
\email{\href{mailto:dliving5@uwyo.edu}{dliving5@uwyo.edu}}

\author[JLM]{Jeremy L.\ Martin}
\address{Department of Mathematics, University of Kansas, United States}
\email{\href{mailto:jlmartin@ku.edu}{jlmartin@ku.edu}}

\author[GDN]{George D.\ Nasr}
\address{Department of Mathematics, University of Nebraska--Lincoln, United States}
\email{\href{mailto:george.nasr@huskers.unl.edu}{george.nasr@huskers.unl.edu}}

\author[JV]{Julianne Vega}
\address{Department of Mathematics, Kennesaw State University, United States}
\email{\href{mailto:jvega30@kennesaw.edu}{jvega30@kennesaw.edu}}

\author[HW]{Hays Whitlatch}
\address{Department of Mathematics, Gonzaga University, United States}
\email{\href{mailto:whitlatch@gonzaga.edu}{whitlatch@gonzaga.edu}}

\subjclass[2010]{%
05C31, 
05C70, 
05E45, 
82B43} 
\keywords{Graph, simplicial complex, Euler characteristic, nucleus, percolation, block-cutpoint tree}
\begin{document}
\date{\today}
\begin{abstract}
Veit Elser proposed a random graph model for percolation in which physical dimension appears as a parameter.  Studying this model combinatorially leads naturally to the consideration of numerical graph invariants which we call \emph{Elser numbers} $\els_k(G)$, where $G$ is a connected graph and $k$ a nonnegative integer.  Elser had proven that $\els_1(G)=0$ for all $G$.  By interpreting the Elser numbers as reduced Euler characteristics of appropriate simplicial complexes called \emph{nucleus complexes}, we prove that for all graphs $G$, they are nonpositive when $k=0$ and nonnegative for $k\geq2$.  The last result confirms a conjecture of Elser.  Furthermore, we give necessary and sufficient conditions, in terms of the 2-connected structure of~$G$, for the nonvanishing of the Elser numbers.
\end{abstract}

\maketitle
\section{Introduction}

Let $G=(V(G),E(G))$ be a connected undirected graph and $k\geq 0$ an integer.  A \defterm{nucleus} of $G$ is a connected subgraph $N\subseteq G$ such that $V(N)$ is a \defterm{vertex cover}; that is, every edge of $G$ has at least one endpoint in $V(N)$.  Let $\N(G)$ denote the set of all nuclei of~$G$.  The \defterm{$k^{\textrm{th}}$ Elser number} of $G$ is
\begin{equation}\label{define-elser}
\els_k(G) = (-1)^{|V(G)|+1} \sum_{N\in\N(G)} (-1)^{|E(N)|} |V(N)|^{k}
\end{equation}
This invariant was introduced by Veit Elser~\cite{Elser}, who conjectured~\cite{Elser-MO} that $\els_k(G)\geq 0$ for all graphs $G$ and integers $k \geq 2$.
In this paper, we answer completely the question of when $\els_k(G)$ is positive, negative or zero.
\begin{theorem}\label{thm:els}
Let $G$ be a connected graph with at least two vertices.  Then:
\begin{enumerate}[label=(\alph*)]
\item\label{mainthm:0} $\els_0(G) \leq 0$.
\item\label{mainthm:1} $\els_1(G)=0$.
\item\label{mainthm:2} $\els_k(G)\geq 0$ for all integers $k\geq 2$.  That is, Elser's conjecture holds.
\end{enumerate}
\end{theorem}

Part~\ref{mainthm:1} is \cite[Theorem~2]{Elser}.  Theorem~\ref{thm:els} extends  \cite[Theorem~2]{Elser} to all $k$. We also extend the previous  result: a characterization of strict positivity of the Elser numbers.

\begin{theorem}\label{thm:strict-positivity}
Let $G$ be a connected simple graph.
\begin{enumerate}[label=(\alph*)]
\item\label{strict-pos:2conn} If $G$ has no cut-vertex, then $\els_0(G)<0$, $\els_1(G)=0$, and $\els_k(G) > 0$ for all $k \geq 2$.
\item\label{strict-pos:not2conn} Otherwise, $\els_k(G) \neq 0$ if and only if $k \geq \ell$, where $\ell\geq 2$ is the number of leaves in the block-cutpoint tree of~$G$ (that is, the number of 2-connected components of $G$ that contain exactly one cut-vertex).
\end{enumerate}
\end{theorem}

Before describing the methods of proof, we describe the motivation behind Elser's conjecture, which arises in percolation theory.  Roughly speaking, percolation models a physical medium by a random graph $\Gamma$, often taken to be a subgraph of $\Z^2$ or some other periodic lattice.  Vertices or edges occur independently with some fixed probability, corresponding to the presence or absence of atoms or bonds between them, and the permeability of the medium is modeled by the component structure of the graph.  For an overview of percolation theory, see the excellent expository article by Kesten~\cite{Kesten}.  In many percolation models, the ambient graph (such as $\Z^2$) controls the combinatorics so strongly that one cannot consider physical dimension as a parameter of the model, but must study different dimensions as separate problems.

For this reason, Elser~\cite{Elser} proposed a percolation model in which dimension can be treated as a parameter, following work of Gaunt and Fischer \cite{FG} and Leibbrandt \cite{Leibb}. 
Elser's model starts with a random geometric graph model consisting of a collection of $N$ points uniformly distributed throughout a $D$-dimensional volume $V$.  The edge between two points $z_1,z_2$ occurs with probability $\exp(-a\|z_1-z_2\|^2)$, where $a$ is some fixed constant.  Let $n_k=n_k(a,V,z_1,\dots,z_N)$ be the expected number of \emph{$k$-clusters}, or connected components with $k$ vertices.  Using a property of Gaussian integrals due to Kirchhoff~\cite{Kirchhoff}, Elser expanded the generating function for the numbers $n_k$ as
\[
\mathsf{F}(x,y)
= \sum_{k=1}^\infty  y^k n_k
= \sum_{m=1}^\infty \frac{x^{m-1}}{m!} \left(\sum_{G\in\mathscr{C}_m} \left(\frac{1}{\tau(G)}\right)^{D/2}W(G,y)
\right)
\]
\cite[eqn.~(6)]{Elser}, where $\mathscr{C}_m$ denotes the set of simple connected graphs on $m$ labeled vertices; $\tau(G)$ the number of spanning trees of $G$; and
\[
W(G,y) = \sum_{N\in\mathcal{N}(G)} (-1)^{|E(G)|-|E(N)|}y^{|V(N)|}
\]
where as before $\mathcal{N}(G)$ is the set of nuclei of $G$.  What we call the $k^{th}$ Elser number equals $(yD_y)^kW(G,y)\vert_{y=1}$, where $D_y$ means differentiation with respect to~$y$.

We prove Elser's conjecture using techniques from topological combinatorics.  Our general approach is to interpret the numbers $\els_k(G)$ as sums of reduced Euler characteristics $\tilde\chi(\Delta^G_U)$.  Here $\Delta^G_U$ is a simplicial complex whose faces correspond to nuclei containing a specified set $U$ of vertices.  The precise formula is given by Theorem~\ref{thm:elser-from-summands} below.  While the topology of these simplicial complexes remains mysterious in most cases, it is nonetheless possible to establish a deletion/contraction-type recurrence for their reduced Euler characteristics (Theorem~\ref{thm:delete-contract}) and thus to determine precisely the sign of $\tilde\chi(\Delta^G_U)$, which turns out to be just $(-1)^{|E(G)|+|V(G)|}$ (Proposition~\ref{prop:euler-all}).  The upshot is that every summand in the expression for $\els_k(G)$ in Theorem~\ref{thm:elser-from-summands} is nonnegative, proving Theorem \ref{thm:els}.

Most of our arguments phrased topologically can be replaced with purely combinatorial arguments using techniques such as inclusion/exclusion or sign-reversing involutions (for example, the topological statement that cones are contractible, hence have zero reduced Euler characteristic, can be replaced with the combinatorial statement that toggling the cone point is a fixed-point-free involution on faces that changes parity of dimension).  We originally approached Elser's problem by investigating the topology of nucleus complexes (see \S\ref{sec:future}), where there are still open problems to resolve.

The paper is structured as follows.  In Section~\ref{sec:prelim} we set up notation for graphs and simplicial complexes, and give basic definitions and facts about nuclei and Elser numbers.  The proofs of the main theorems occupy Sections~\ref{sec:construct-complex}--\ref{sec:positivity} of the paper.  In Section~\ref{sec:construct-complex}, we construct the simplicial complexes $\Delta^G_U$ and prove the first results linking the Elser numbers to their reduced Euler characteristics; the deletion-contraction recurrence is established in Section~\ref{section:deletion-contraction}.  Using these tools, we then prove Elser's conjecture for trees in Section~\ref{sec:trees} and for general graphs in Section~\ref{sec:proof-of-conj}.  Theorem~\ref{thm:strict-positivity} is proved in Section~\ref{sec:positivity}.  Its proof requires a purely graph-theoretic result (Theorem~\ref{all-ears}) that strengthens the standard result that every 2-connected graph has an ear decomposition, and may be of independent interest (see~\cite{Schmidt} for a related algorithm).  Section~\ref{sec:monotonicity} proves a monotonicity result: $\els_k(G) \geq \els_k(G/e)+\els_k(G\sm e)$ for all~$G$ and~$e$, with equality when $k=0$.  We conclude in  Section~\ref{sec:future} with observations and conjectures on the topology of nucleus complexes, which appear to have a rich structure.

\section{Preliminaries} \label{sec:prelim}

\subsection{Graphs, nuclei, and Elser numbers}

As a general reference for the graph theory necessary for this paper, we refer the reader to \cite[Sections 1.1, 1.3, 2.1, 4.1, 4.2]{West}.
Throughout, ``graph'' means ``undirected graph.''  The vertices and edges of a graph are denoted by $V(G)$ and $E(G)$ respectively. For $A \subseteq E(G)$, we write $\ov{A} = E(G)\sm A$ for the complement of~$A$, if the ambient graph $G$ is clear from context.

Two edges are \defterm{parallel} if they have the same pair of endpoints (or are both loops incident to the same vertex).  We write $\Par(e)$ for the equivalence class of all edges parallel to $e$.  The \defterm{deparallelization} $\Dep(G)$ is the graph obtained from $G$ by identifying all edges in the same parallel class.

The \defterm{deletion} of an edge $e$ from $G$ is the graph $G \sm e$ with vertex set $V(G)$ and edge set $E(G) \sm \{e\}$. The \defterm{contraction} of $e$ in $G$ is the graph $G/e$ obtained by removing $e$ and identifying its endpoints $v,w$ into a single vertex (denoted $vw$). A \defterm{minor} of $G$ is a graph obtained by some sequence of deletions and contractions, i.e., of the form $G/C\sm D$, where $C,D$ are disjoint subsets of $E(G)$.  Every $U\subseteq V(G)$ gives rise to a set $U/e\subseteq V(G/e)$, for an edge $e = \{u,v\}$, defined by
\[U/e=\begin{cases}
U & \text{ if } v,w\not\in U,\\
U\sm\{v,w\}\cup\{vw\} &\text{ otherwise.}
\end{cases}\]
This notation can be iterated; if $C=\{e_1,\dots,e_k\}\subseteq E(G)$, then we set $U/C=((U/e_1)/\cdots)/e_k$; the order of contraction does not matter.  If $H=G/C\sm D$ is a minor of $G$, then we write $U[H]$ for $U/C$.

An edge $e\in E(G)$ is a \defterm{cut-edge} of $G\sm e$ has more components than $G$.  Likewise, a vertex $x\in V(G)$ is a \defterm{cut-vertex} if $G-x$ has more components than $x$, where $G-x$ is the graph obtained by deleting $x$ and all incident edges.

A \defterm{vertex cover} of $G$ is a set $C \subseteq V(G)$ such that every edge $e \in E(G)$ has at least one endpoint in $C$. In particular, if $G$ has a loop at vertex $v$, then every vertex cover of $G$ must contain $v$. Notice that the vertex covers of $\Dep(G)$ are the same as those of $G$.

\begin{definition}
A \defterm{nucleus} of $G$ is a connected subgraph $N$ of $G$ whose vertices $V(N)$ form a vertex cover of $G$. We denote the set of nuclei of $G$ by $\N(G)$.
\end{definition}

Note that Elser assumed that $G$ is simple, which is most natural from a physical point of view; however, we do not make this assumption, since non-simple graphs will naturally arise.

\begin{prop}\label{prop:cutsets}
Let $N \in \N(G)$. Let $C \subseteq V(G)$, and suppose $G \sm C$ is disconnected. Then $V(N) \cap C \neq \0$.\\
In particular, $V(N)$ contains all cut vertices of $G$.
\end{prop}
\begin{proof}
Suppose $V(N) \cap C$ is empty. Since $V(N)$ is a vertex cover, $V(N)$ must contain all neighbors of vertices in $C$. In particular, $N$ contains two vertices in different components of $G \sm C$. But since $N=N \sm C$, this implies $N$ is not connected, a contradiction.
\end{proof}

\begin{example}\label{ex:K2}
The complete graph $K_2$ on two vertices has three nuclei: itself and its two one-vertex subgraphs.  Therefore,
\[
\els_k(K_2) = (-1)^{2 +1}\sum\limits_{N \in \N(K_2)}(-1)^{|E(N)|}|V(N)|^k = -1(1 +1 -2^k) = 2^k-2.
\]
\end{example}

\begin{example}\label{ex:closed-formulas}
For many standard graphs, it is easy to determine their nuclei and Elser numbers.
\begin{enumerate}[label=(\alph*)]
\item\label{ex:tree} Let $T$ be a tree with $n \geq 3$ vertices.  Then its nuclei are precisely the subgraphs obtained by deleting some set of leaf vertices.  In particular, if $T$ has $\ell$ leaves, then it has $2^\ell$ nuclei. Moreover, if $L$ is the set of leaf vertices in $T$, then
\[\els_k(T) = (-1)^{n+1} \sum_{J\subseteq L} (-1)^{n-|J|-1} |V(T) \sm J|^k = \sum_{j=0}^{\ell} (-1)^{\ell+j} \binom{\ell}{j} (n-j)^k.\]
\item As a special case, for $n\geq3$, the $n$-vertex path $P_n$ has four nuclei: itself and the paths obtained by deleting one or both endpoints.  So:
\[
\els_k(P_n) = n^k-2(n-1)^k+(n-2)^k.
\]
\item\label{ex:C3} The cycle graph $C_n$ has precisely $2n+1$ nuclei: itself, the $n$ copies of $P_n$ obtained by deleting a single edge, and the $n$ copies of $P_{n-1}$ obtained by deleting a single vertex and its two incident edges. For example, here are the seven nuclei of $C_3$:

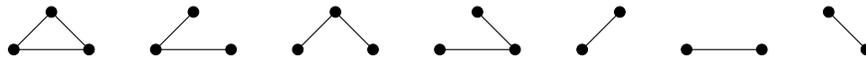
\begin{center}
\begin{tikzpicture}[scale=.5]
\node[point] (A) at (1,1){};
\node[point] (B) at (0,0){};
\node[point]  (C) at (2,0){};

\path[-] (A) edge (B);
\path[-] (B) edge (C);
\path[-] (A) edge (C);
 \end{tikzpicture}
 \hspace{.5 cm}
 \begin{tikzpicture}[scale=.5]
\node[point] (A) at (1,1){};
\node[point] (B) at (0,0){};
\node[point]  (C) at (2,0){};

\path[-] (A) edge (B);
\path[-] (B) edge (C);
 \end{tikzpicture}
 \hspace{.5 cm}
 \begin{tikzpicture}[scale=.5]
\node[point] (A) at (1,1){};
\node[point] (B) at (0,0){};
\node[point]  (C) at (2,0){};

\path[-] (A) edge [bend right =0] node[above] {} (B);
\path[-] (A) edge [bend right =0] node[above] {} (C);
 \end{tikzpicture}
 \hspace{.5 cm}
 \begin{tikzpicture}[scale=.5]
\node[point] (A) at (1,1){};
\node[point] (B) at (0,0){};
\node[point]  (C) at (2,0){};
\path[-] (B) edge (C);
\path[-] (A) edge (C);
 \end{tikzpicture}
 \hspace{.5 cm}
 \begin{tikzpicture}[scale=.5]
\node[point] (A) at (1,1){};
\node[point] (B) at (0,0){};

\path[-] (A) edge [bend right =0] node[above] {} (B);
 \end{tikzpicture}
 \hspace{.5 cm}
 \begin{tikzpicture}[scale=.5]
\node[point] (B) at (0,0){};
\node[point]  (C) at (2,0){};

\path[-] (B) edge (C);
 \end{tikzpicture}
 \hspace{.5 cm}
 \begin{tikzpicture}[scale=.5]
\node[point] (A) at (1,1){};
\node[point]  (C) at (2,0){};

\path[-] (A) edge(C);
 \end{tikzpicture}
\end{center}

Thus the Elser numbers are:
\begin{align*}
\els_k(C_n)
&=  (-1)^{n+1}\Big((-1)^{n}n^{k}+n\left((-1)^{n-1}n^{k}\right)+n\left((-1)^{n-2}(n-1)^{k}\right) \Big) \notag\\
&=  n(n-1)\left( n^{k-1}-(n-1)^{k-1} \right). \label{elser-cycle}
\end{align*}
When $n=3$, this reduces to
 \(
 \els_k(C_3)  = 6~(3^{k-1}-2^{k-1}).
 \)
\end{enumerate}
\end{example}

\subsection{Simplicial complexes}

We will study the nuclei of a graph using the language of simplicial complexes and their Euler characteristics, which we now introduce briefly.  An \defterm{(abstract) simplicial complex} $\Delta$ on a finite set $X$ of \defterm{vertices} is a set of subsets of $X$ (called \defterm{faces}) such that
\begin{enumerate}
\item[(i)] $\emptyset \in \Delta$;
\item[(ii)] If $\sigma \in \Delta$ and $\tau \subseteq \sigma$, then $\tau \in \Delta$.
\end{enumerate}
The \defterm{reduced Euler characteristic} of $\Delta$ is
\begin{equation} \label{comb-euler}
\tilde\chi(\Delta)=\sum_{\sigma\in\Delta} (-1)^{\dim\sigma}
\end{equation}
where $\dim\sigma=|\sigma|-1$.  For example, $\Delta$ is said to be a \defterm{cone} with \defterm{cone point}
$x$ if every maximal face contains $x$; in this case toggling the cone point gives a sign-reversing involution on $\Delta$, so $\tilde\chi(\Delta)=0$.

These primitive notions largely suffice for the techniques used in the paper, except for \S\ref{sec:future}.  For the interested reader, we give a very brief summary of the topology of simplicial complexes.  (Experts will note that we omit several refinements, such as homology groups with arbitrary coefficients.)  For the complete story, we refer the reader to \cite[\S 2.1]{Hatcher} or \cite[\S2.3]{Kozlov}.  

An abstract simplicial complex can be regarded as a topological space in the following way.  Assume that $X=\{1,\dots,n\}$, and associate with each face $\sigma\in\Delta$ the convex hull of the standard basis vectors $\{\mathbf{e}_i:\ i\in\sigma\}$; the \defterm{standard geometric realization} $|\Delta|$ of $\Delta$ is the union of all such simplices.  (Note that the convex hull of $d$ points has geometric dimension~$d-1$, explaining the definition of dimension above.)  More generally, a \defterm{geometric realization} of $\Delta$ is any topological space homeomorphic to~$|\Delta|$.  Topological invariants of $|\Delta|$, such as its singular homology groups, can be obtained from the purely combinatorial structure of $\Delta$.  Thus it makes sense to say that $\Delta$ itself has topological properties, such as contractibility.  It is frequently convenient to make no distinction between an (abstract) simplicial complex and its geometric realization.

In one important special case (nucleus complexes of graphs with two vertices) we will need the slightly more general notion of a \defterm{$\Delta$-complex} (also known as a \textit{trisp} or \textit{triangulated space}).  A $\Delta$-complex is much like a simplicial complex in that it is built out of simplices attached to each other along common sub-simplices.  However, we no longer require that distinct simplices have distinct vertex sets, or even that each $(d-1)$-dimensional simplex have $d$ distinct vertices.  For example, all graphs are 1-dimensional $\Delta$-complexes, while only simple graphs (those with no loops or parallel edges) are simplicial complexes.  The reduced Euler characteristic and other topological properties of $\Delta$-complexes are defined the same way as for simplicial complexes.

The homology groups of a simplicial or $\Delta$-complex are fundamental invariants that measure, among other things, the number of ``holes'' of $\Delta$ of various dimensions.  The topological boundary $\bd\sigma$ of every $k$-dimensional face $\sigma\in\Delta$ is a union of its $(k-1)$-faces; algebraically, the boundary operation gives rise to maps $\bd_k:C_k(\Delta;\R)\to C_{k-1}(\Delta;\R)$, where $C_k(\Delta;\R)$ is the $\R$-vector space spanned by $k$-faces of~$\Delta$.  The signs are arranged in a way that keeps track of relative orientation, and have the consequence that $\bd_k\bd_{k+1}$ is the zero map for all $k$; equivalently, $\im\bd_{k+1}\subseteq\ker\bd_k$.  The \defterm{$k^{th}$ reduced simplicial homology group} of $\Delta$ is $\HH_k(\Delta;\R)=(\ker\bd_k)/(\im\bd_{k+1})$.  These groups turn out to be topological invariants of the geometric realization $|\Delta|$.  In particular, the reduced Euler characteristic of $|\Delta|$ is
$\sum_{n \geq 0} (-1)^n\dim_\R(\HH_n(\Delta;\R))$; the equality between this formula and the purely combinatorial formula~\eqref{comb-euler} is known as the \textit{Euler-Poincar\'e theorem}.  Note that if $\Delta$ is a cone, then $|\Delta|$ is contractible (because it deformation-retracts onto the cone point), hence all homology groups vanish, confirming that $\tilde\chi(\Delta)=0$.

\section{Nucleus complexes}
\label{sec:construct-complex}

In this section, we study the \defterm{nucleus complexes} of a graph $G$.  For each $U\subseteq V(G)$, the $U$-nucleus complex $\Delta_U^G$ is a simplicial complex whose vertices are the edges of $G$; its faces are complements of nuclei whose vertex support contains~$U$.  We show that the $k^{th}$ Elser number of $G$ may be written as a weighted sum of reduced Euler characteristics of nucleus complexes (Theorem~\ref{thm:elser-from-summands}).

Elser notes the following identity \cite[Proof of Theorem~2]{Elser} :
\begin{align*}
\els_1(G)
&= (-1)^{|V(G)|+1}\sum_{N\in\N(G)} (-1)^{|E(N)|}|V(N)|\\
&= (-1)^{|V(G)|+1}\sum_{v \in V(G)} \sum_{\substack{N\in\N(G):\\ v \in V(N)}} (-1)^{|E(N)|}.
\end{align*}
This identity allowed Elser to characterize $\els_1(G)$ for any $G$. We give a more general identity, which works for any $k \geq 0$. Let $\Sur(a,b)$ denote the number of surjections from a set of size $a$ to a set of size $b$. (By convention, we set $\Sur(0,0)=1$ and $\Sur(a,b)=0$ if exactly one of $a,b$ is zero.)  (Note that $\Sur(a,b)=b!\,S(a,b)$, where $S(a,b)$ denotes a Stirling number of the second kind.)

\begin{prop}\label{prop:elser-from-summands:0}
Let $G$ be a graph and $k$ a nonnegative integer. Then
\[\els_k(G)=(-1)^{|E(G)|+|V(G)|+1}\sum_{U \subseteq V(G)} \Sur(k,|U|) \sum_{\substack{N\in\N(G):\\ U\subseteq V(N)}} (-1)^{|E(\ov{N})|}.\]
\end{prop}
\begin{proof}
The term $|V(N)|^k$ counts functions $[k]\to V(N)$, and such a function is the same thing as a surjection from $[k]$ to some subset of $V(N)$.  Therefore,
\begin{align*}
\els_k(G)
&= (-1)^{|V(G)|+1}\sum_{N\in\N(G)} (-1)^{|E(N)|} \sum_{U \subseteq V(N)} \Sur(k,|U|)\\
&= (-1)^{|E(G)|+|V(G)|+1}\sum_{U \subseteq V(G)} \Sur(k,|U|) \sum_{\substack{N\in\N(G):\\ U\subseteq V(N)}} (-1)^{|E(\ov{N})|}.\qedhere
\end{align*}
\end{proof}

We now rephrase Proposition~\ref{prop:elser-from-summands:0} in terms of Euler characteristics of nucleus complexes.

\begin{definition} \label{defn:u-nucleus-complex}
Let $G$ be a connected graph with $|V(G)|\geq 3$, and let $U\subseteq V(G)$.  The \defterm{$U$-nucleus complex of $G$} is the simplicial complex
\(
\Delta^G_U
=\{E(G)\sm E(N):\ N\in\N(G),\ V(N)\supseteq U\}.
\)
The set $\Delta^G_U$ is a simplicial complex because every graph obtained by adding edges to a nucleus is also a nucleus.
\end{definition}

\begin{example}
Label the vertices of $K_3$ as 1, 2, 3 and its edges as 12, 13, 23.  The nuclei of $K_3$ are shown in Example~\ref{ex:closed-formulas}\ref{ex:C3}.  Accordingly, its nucleus complexes $\Delta^{K_3}_U$ are as shown in Figure~\ref{fig:nucleus-K3}.  Up to isomorphism, the complex $\Delta^{K_3}_U$ depends only on $|U|$.

\begin{figure}[ht]
\begin{center}
\begin{tikzpicture}[scale=0.8]
\newcommand{\trisize}{.8}
\newcommand{\lableng}{1.2}
\newcommand{\spacing}{4}
\newcommand{\vxsize}{0.075}
\draw (90:\trisize)--(210:\trisize)--(330:\trisize)--cycle;
\draw[fill=black] (90:\trisize) circle(\vxsize);	\node at (90:\lableng) {\scriptsize\sf12};
\draw[fill=black] (210:\trisize) circle(\vxsize);	\node at (210:\lableng) {\scriptsize\sf13};
\draw[fill=black] (330:\trisize) circle(\vxsize);	\node at (330:\lableng) {\scriptsize\sf23};
\node at (270:\lableng+.4) {$U=\0$};
\node at (270:\lableng+1.2) {$\tilde\chi(\Delta^{K_3}_U)=-1$};
\begin{scope}[shift={(\spacing,0)}]
\draw (90:\trisize)--(330:\trisize)--(210:\trisize);
\draw[fill=black] (90:\trisize) circle(\vxsize);	\node at (90:\lableng) {\scriptsize\sf12};
\draw[fill=black] (210:\trisize) circle(\vxsize);	\node at (210:\lableng) {\scriptsize\sf13};
\draw[fill=black] (330:\trisize) circle(\vxsize);	\node at (330:\lableng) {\scriptsize\sf23};
\node at (270:\lableng+.4) {$U=\{1\}$};
\node at (270:\lableng+1.2) {$\tilde\chi(\Delta^{K_3}_U)=0$};
\end{scope}
\begin{scope}[shift={(2*\spacing,0)}]
\draw (210:\trisize)--(330:\trisize);
\draw[fill=black] (90:\trisize) circle(\vxsize);	\node at (90:\lableng) {\scriptsize\sf12};
\draw[fill=black] (210:\trisize) circle(\vxsize);	\node at (210:\lableng) {\scriptsize\sf13};
\draw[fill=black] (330:\trisize) circle(\vxsize);	\node at (330:\lableng) {\scriptsize\sf23};
\node at (270:\lableng+.4) {$U=\{1,2\}$};
\node at (270:\lableng+1.2) {$\tilde\chi(\Delta^{K_3}_U)=1$};
\end{scope}
\begin{scope}[shift={(3*\spacing,0)}]
\draw[fill=black] (90:\trisize) circle(\vxsize);	\node at (90:\lableng) {\scriptsize\sf12};
\draw[fill=black] (210:\trisize) circle(\vxsize);	\node at (210:\lableng) {\scriptsize\sf13};
    \draw[fill=black] (330:\trisize) circle(\vxsize);	\node at (330:\lableng) {\scriptsize\sf23};
\node at (270:\lableng+.4) {$U=\{1,2,3\}$};
\node at (270:\lableng+1.2) {$\tilde\chi(\Delta^{K_3}_U)=2$};
\end{scope}
\end{tikzpicture}
\end{center}
\caption{The nucleus complexes $\Delta^{3K_2}_U$.\label{fig:nucleus-K3}}
\end{figure}
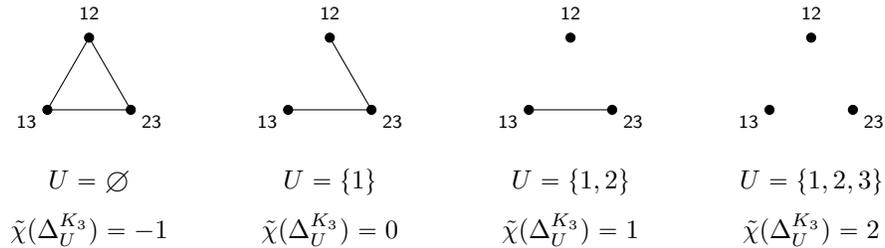
\end{example}

The definition of nucleus complexes when $|V(G)|=2$ requires special handling, for the following reason.  Consider the graph $G=cK_2$ with two vertices $v_1,v_2$ and $c>0$ parallel edges.  The subtlety is that $cK_2$ has two distinct nuclei with the same edge sets, namely the subgraphs $N_1,N_2$ with $V(N_i)=\{v_i\}$ and $E(N_i)=\0$.  Accordingly, we define $\Delta^{cK_2}_\0$ to be the $\Delta$-complex consisting of two $(c-1)$-dimensional simplices $\sigma_1,\sigma_2$ on vertex set $E(cK_2)$, glued along their boundaries (which are $(c-2)$-spheres) to produce a $(c-1)$-sphere; see Figure~\ref{fig:nucleus-K2} for the cases $c=2$ and $c=3$.  Each simplex $\sigma_i$ should be regarded as recording the complement $E(N_i)\sm E(cK_2)$.  This construction is necessary to preserve the correspondence between nuclei of $G$ and faces of $\Delta^G_\0$.  For $U\neq\0$, we can define $\Delta^{cK_2}_U$ just as in Definition~\ref{defn:u-nucleus-complex}: in particular,
\begin{align}
\Delta^{cK_2}_\0&\cong\mathbb{S}^{c-1},&
\Delta^{cK_2}_{\{v_1\}}&=\sigma_1,&
\Delta^{cK_2}_{\{v_2\}}&=\sigma_2,&
\Delta^{cK_2}_{\{v_1,v_2\}}&=\partial\sigma_1=\partial\sigma_2\cong\mathbb{S}^{c-2}, \label{K2:Delta}
\intertext{where $\cong$ means homeomorphism; $\mathbb{S}^k$ means the $k$-dimensional sphere; and $\partial \sigma_i$  is the topological boundary of $\sigma_i$, namely $ \bigcup \limits_{\gamma \subsetneq \sigma_i} \gamma$.  Therefore,}
\tilde\chi(\Delta^{cK_2}_\0)&=(-1)^{c-1},&
\tilde\chi(\Delta^{cK_2}_{\{v_1\}})&=0,&
\tilde\chi(\Delta^{cK_2}_{\{v_2\}})&=0,&
\tilde\chi(\Delta^{cK_2}_{\{v_1,v_2\}})&= (-1)^c. \label{K2:euler}
\end{align}

\begin{figure}[ht]
\begin{center}
\begin{tikzpicture}
\newcommand{\vxsize}{.075}

\foreach \x/\lab in {0/1,1/2} {
    \draw[fill=black] (\x,0) circle (\vxsize);
    \node at (\x,-.5) {$\sigma_\lab$};
}

\begin{scope}[shift={(6,0)}]
\draw (2,0) arc (0:360:2 and 0.6);
\foreach \x in {-2,2} \draw[fill=black] (\x,0) circle (\vxsize);
\node at (-1.5,.75) {$\sigma_1$};
\node at (1.5,-.75) {$\sigma_2$};
\end{scope}

\begin{scope}[shift={(12,0)}]
\shade[ball color = gray!40, opacity = 0.4] (0,0) circle (2cm);
\draw (0,0) circle (2cm);
\draw (-2,0) arc (180:360:2 and 0.6);
\draw[dashed] (2,0) arc (0:180:2 and 0.6);
\foreach \x/\y in {0/-.6, .8*2/.6*.6,  -.8*2/.6*.6} \draw[fill=black] (\x,\y) circle (\vxsize);
\node at (-.8,1.5) {$\sigma_1$};
\node at (.8,-1.5) {$\sigma_2$};
\end{scope}
\end{tikzpicture}
\caption{The nucleus complexes $\Delta^{K_2}_{\0}$ (left), $\Delta^{2K_2}_{\0}$ (middle), and $\Delta^{3K_2}_{\0}$ (right).\label{fig:nucleus-K2}}
\end{center}
\end{figure}
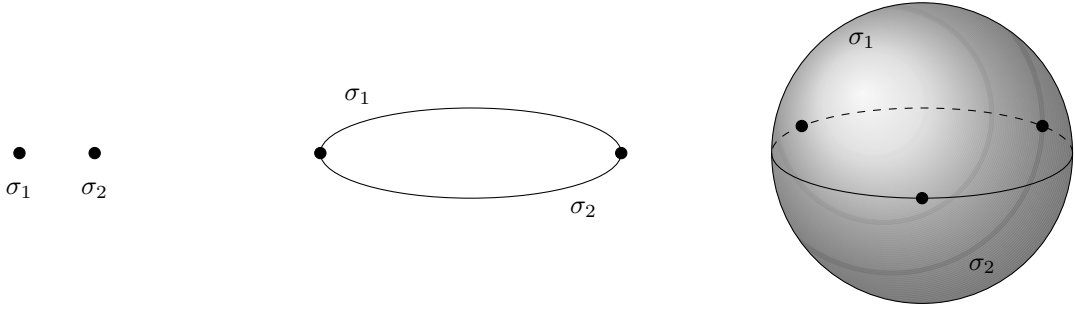
The inner sum over nuclei in Proposition~\ref{prop:elser-from-summands:0} is just $-\tilde\chi(\Delta^G_U)$, so we can rewrite Proposition~\ref{prop:elser-from-summands:0} to give a formula for Elser numbers in terms of Euler characteristics:

\begin{theorem} \label{thm:elser-from-summands}
Let $G$ be a graph and $k\geq0$ an integer.  Then
\[
\els_k(G) = (-1)^{|E(G)|+|V(G)|} \sum_{U\subseteq V(G)} \Sur(k,|U|) \;\tilde\chi(\Delta^G_U).
\]
\end{theorem}

Nucleus complexes are well-behaved with respect to loops and cut-edges, at least at the level of Euler characteristic. Let $G\neq K_2$ be a graph and $e\in E(G)$ be a cut-edge. If one of the endpoints $x$ of $e$ has degree~1, we say that $e$ is a \defterm{leaf edge (with leaf $x$)}.

\begin{prop}
\label{elser-complex-loops-and-parallels}
Let $G$ be a graph and $U\subseteq V(G)$.
\begin{enumerate}[label=(\alph*)]
\item \label{elser-complex-loops}
If $G$ has a loop $\ell$, then $\Delta^G_U$ is a cone with cone point $\ell$.  In particular, $\tilde\chi(\Delta^G_U)=0$.
\item \label{elser-complex-parallels}
Let $D=\Dep(G)$.  Then $\tilde\chi(\Delta^D_U)=(-1)^{|E(G)|-|E(D)|}\tilde\chi(\Delta^G_U)$.
\item \label{elser-complex-leaves}
Suppose that $G\neq K_2$ and that $e$ is a cut-edge of $G$.  If $e$ is a leaf edge with leaf $x$ and $x\notin U$, then
$\Delta^G_U$ is a cone.  Otherwise, $\Delta^G_U=\Delta^{G/e}_{U/e}$.
\end{enumerate}
\end{prop}
\begin{proof}

\ref{elser-complex-loops} Let $s\in V(G)$ be the vertex incident to $\ell$. The vertex set of every nucleus $N\in\mathcal{N}(G)$ must contain $s$, regardless of $U$. So for all $N\in\mathcal{N}(G)$, let $N'$ be the subgraph of $G$ induced by the edges $E(N)\cup\{\ell\}$. Then $N'\in\mathcal{N}(G)$ and it follows that $\Delta_U^G$ is a cone with cone point $\ell$. Since every cone is contractible, the reduced Euler characteristic is zero.

\ref{elser-complex-parallels} By induction, it suffices to show that if $a,b$ are parallel edges in $G$ and $G'=G-b$, then
\[
\tilde\chi(\Delta^{G'}_U)=-\tilde\chi(\Delta^G_U)
\]
for every $U\subseteq V(G)$. Let
\begin{align*}
\A_0&=\{ A\in\Delta^G_U:\ a,b\not\in A\}, &
\A_a&=\{ A\in\Delta^G_U:\ a\in A,\ b\not\in A\},\\
\A_b&=\{ A\in\Delta^G_U:\ a\not\in A,\ b\in A\}, &
\A_{ab}&=\{ A\in\Delta^G_U:\ a,b\in A\}.
\end{align*}
Toggling~$a$ gives a bijection between $\A_b$ and $\A_{ab}$, so
\begin{align*}
\tilde\chi(\Delta^G_U)
&= \sum_{ A\in\A_0} (-1)^{|E(G)|-|A|}
+ \sum_{ A\in\A_a} (-1)^{|E(G)|-|A|}
+ \sum_{ A\in\A_b} (-1)^{|E(G)|-|A|}
+ \sum_{ A\in\A_{ab}} (-1)^{|E(G)|-|A|}\\
&= \sum_{ A\in\A_0} (-1)^{|E(G)|-|A|}
+ \sum_{ A\in\A_a} (-1)^{|E(G)|-|A|} \\
&= \sum_{\substack{ A\in\Delta^{G'}_U\\ a\not\in A}} (-1)^{|E(G')|+1-|A|}
+ \sum_{\substack{ A\in\Delta^{G'}_U\\ a\in A}} (-1)^{|E(G')|+1-|A|}\\
&= -\tilde\chi(\Delta^{G'}_U).
\end{align*}

\ref{elser-complex-leaves} Let $x,y$ be the endpoints of the cut-edge $e$.  Suppose that $e$ is not a leaf edge.  Then it must belong to every nucleus in $G$, because every vertex cover must include at least one vertex from each component of $G-e$.  On the other hand, every nucleus $N$ in $G/e$ must include at least one edge in each cut-component of the fused vertex $xy$, and since $N$ is connected we must have $xy\in V(N)$.  Therefore, $\Delta^G_U=\Delta^{G/e}_{U/e}$ for all $U$.

Now, suppose that $e$ is a leaf edge with leaf $x$.  Every nucleus must include $y$ in its vertex set, and toggling $e$ does not change whether an edge set is a nucleus.  Therefore, if $x\notin U$, then $\Delta^G_U$ is a cone with cone point $e$, hence has reduced Euler characteristic 0.  If $x\in U$ then every $U$-nucleus must include the edge $e$, so $\Delta^G_U=\Delta^{G/e}_{U/e}$.
\end{proof}

Parts~\ref{elser-complex-loops} and~\ref{elser-complex-parallels} have immediate consequences for Elser numbers, which we now state as a corollary.  (Part~\ref{elser-complex-leaves} will be useful in computing Elser numbers for trees in the next section.)

\begin{cor}\label{cor:simplify-els}
Let $G$ be a graph.
\begin{enumerate}[label=(\alph*)]
\item If $G$ contains a loop, then $\els_k(G)=0$ for all $k$.
\item For all $k$, we have $\els_k(G)=\els_k(\Dep(G))$.
\end{enumerate}
\end{cor}

\section{Elser numbers for trees}
\label{sec:trees}

In Example~\ref{ex:closed-formulas}~\ref{ex:tree}, we obtained a formula for the Elser numbers of a tree that depends only on the numbers of vertices and leaves.  This formula has the disadvantage that its sign is not obvious.  On the other hand, we can use Theorem~\ref{thm:elser-from-summands} to give a formula for $k\geq 1$ which is obviously nonnegative.

\begin{prop} \label{prop:tree-euler}
Let $T$ be a tree with two or more vertices, let $U\subseteq V(T)$, and let $L$ denote the set of leaves of $T$. Then:
\[\tilde\chi(\Delta^T_U) = \begin{cases}
1 & \text{ if } T=K_2\text{ and } |U|=0,\\
0 & \text{ if } T=K_2\text{ and }|U|=1,\\
-1 & \text{ if } T=K_2\text{ and }|U|=2,\\
0 & \text{ if } T\neq K_2\text{ and }L\not\subseteq U,\\
-1 & \text{ if } T\neq K_2\text{ and }L\subseteq U.
\end{cases}\]
\end{prop}

\begin{proof}
The first three cases are a restatement of~\eqref{K2:euler}.  On the other hand, suppose that $|V(T)|\geq3$.  If $L \not\subseteq U$, Proposition~\ref{elser-complex-loops-and-parallels}~\ref{elser-complex-leaves} implies $\Delta^T_U$ is a cone and therefore $\tilde\chi(\Delta_U^T) = 0$. When $L \subseteq U$, the only connected subgraph of $T$ containing $L$ is $T$ itself. Thus $\Delta^G_U=\{\emptyset\}$  and so the reduced Euler characteristic is $-1$.
\end{proof}

Ultimately we will reduce the general graph problem to the case of tree graphs, so Proposition~\ref{prop:tree-euler} will be crucial for the proof of Elser's conjecture.  We now give a formula for $\els_k(T)$ when $T$ is a tree.

\begin{cor}\label{cor:elser-for-trees}
Let $k \geq 1$. Let $T$ be a tree with $n$ vertices and $\ell$ leaves. Then
\[
\els_k(T) =\displaystyle\sum_{i=0}^{n-\ell}\binom{n-\ell}{i}\Sur(k,\ell+i)
\]
In particular, Elser's conjecture is true for trees. That is, $\els_k(T)\geq 0$.
\end{cor}
\begin{proof}
Theorem \ref{thm:elser-from-summands} gives
\[
\els_k(T) = (-1)^{|E(T)|+|V(T)|}\sum_{U\subseteq V(T)} \Sur(k,|U|)~\tilde{\chi}(\Delta_U^T).
\]
Since $T$ is a tree, we have
\(
|E(T)|+|V(T)| = 2|V(T)|-1
\)
and thus
\begin{align*}
\els_k(T) & = -\sum_{U\subseteq V(T)} \Sur(k,|U|)~\tilde{\chi}(\Delta_U^T)\\
& = \sum_{U\subseteq V(T)} \Sur(k,|U|)~\left[-\tilde{\chi}(\Delta_U^T)\right].
\end{align*}
Let $L$ denote the set of leaves of $T$. By Proposition \ref{prop:tree-euler}, for $U \neq \0$,
\[
-\tilde\chi(\Delta^T_U) = \begin{cases}
0 & \text{ if } L\not\subseteq U,\\
1 & \text{ if } L\subseteq U
\end{cases}.
\]
Then
\begin{align*}
\els_k(T)& =\displaystyle\sum_{\substack{U \subseteq V(T)\\L \subseteq U}} \Sur(k,|U|)\\
&=\displaystyle\sum_{i=0}^{|V(T)|-|L|}\binom{|V(T)|-|L|}{i}\Sur(k,|L|+i).\qedhere
\end{align*}
\end{proof}

\section{A deletion-contraction recurrence for nucleus complexes}
\label{section:deletion-contraction}
In this section, we develop a deletion-contraction recurrence for reduced Euler characteristics of nucleus complexes of an arbitrary connected graph~$G$.  In general, the main technical tool is a simple bijection $\psi_e$ relating the nucleus complexes of $G$, $G/e$, and $G\sm e$.  Special care must be taken for small graphs, because of the difficulty in defining the nucleus complex of $cK_2$.

Let $G$ be a connected graph and $e$ an edge of $G$  which is neither a loop nor a cut-edge.  Define a map $\psi_e: 2^{E(G)} \rightarrow 2^{E(G \sm e)} \dju 2^{E(G/e)}$ as follows:
\[\psi_e(A)=\begin{cases}
A \sm e \subseteq E(G \sm e) & \textrm{ if $e \in A$,}\\
A \subseteq E(G/e) & \textrm{ if $e \not\in A$.}
\end{cases}\]
We will abbreviate $\psi_e$ by $\psi$.  Note that $\psi$ sends complements of nuclei to complements of nuclei and $\psi$ is a bijection with inverse given by
\[\psi^{-1}(B)=\begin{cases}
B\cup e & \textrm{ for $B\subseteq E(G\sm e)$,}\\
B & \textrm{ for $B\subseteq E(G/e)$.}
\end{cases}\]

Notice that we have not assumed that $G$ is a simple graph, only that $e$ is not a loop or cut-edge. In particular, $\psi$ is well-defined and a bijection even if there is another edge in $G$ with the same endpoints as $e$.  The key technical properties of $\psi$ we will need are as follows.

\begin{prop} \label{all-about-psi}
Let $G$ be a graph, let $e=\{x,y\}$ be an edge of $G$ which is neither a loop nor a cut-edge, and let $U\subseteq V(G)$.  Assume that either (i) $|V(G)|\geq 4$, or (ii) $|V(G)|=3$ and $U\neq\0$.  Then:
\begin{align} \label{psi:12}
\psi\left( \Delta^G_U \right) & \subseteq \Delta^{G\sm e}_U\dju\Delta^{G/e}_{U/e},
~\text{and}\\
\label{psi:5}
\{B \in \Delta^{G\sm e}_U: \psi^{-1}(B)\not\in \Delta^G_U\} & = \{B \in \Delta^{G/e}_{U/e}: \psi^{-1}(B)\not\in \Delta^G_U\}.
\end{align}
\end{prop}

The assumption in the proposition avoids the difficulties in defining $\Delta^{cK_2}_\0$, which can arise from contractions if $|V(G)|\geq3$.

\begin{proof}
To prove~\eqref{psi:12}, let $A\in\Delta_U^G$. If $e\in A$, then $\psi(A) = A\backslash e$ and so $E(G\backslash e)\backslash \psi(A) = E(G\backslash e)\backslash (A\backslash e) = E(G)\backslash A$ is a $U$-nucleus of $G$ not containing $e$. On the other hand, if $e\not\in A$ then $e\in E(G)\backslash A$ and $\psi(A) = A$, so
\begin{align*}
E(G/ e)\backslash \psi(A) = E(G\backslash e)\backslash A.
\end{align*}
Since contraction preserves connectedness and the property of being a vertex cover, the set $E(G\backslash e)\backslash A$ is a $(U/e)$-nucleus. Thus
\[
\psi(A)\in
\begin{cases}
\Delta_U^{G\backslash e} & \text{if }e\in A\\
\Delta_{U/e}^{G/e} & \text{if }e\not\in A,
\end{cases}
\]
which proves \eqref{psi:12}.

To prove~\eqref{psi:5}, suppose that $B\subseteq E(G)\sm e$, so that $B$ can be regarded as a set of edges of any of $G$, $G\sm e$, or $G/e$.  Let $\hat B=E(G\sm e)\sm B = E(G)\sm(B\cup e)$, let $W=V_G(\hat B)$, and let $W'=V_{G/e}(\hat B)$.  Then the following conditions are equivalent:

\begin{enumerate}[label=(\alph*)]
\item $B \in \Delta^{G\sm e}_U$ and $\psi^{-1}(B)\not\in \Delta^G_U$.
\item $B \in \Delta^{G\sm e}_U$ and $B \cup e\not\in \Delta^G_U$.
\item $\hat B$ is a $U$-nucleus of $G\sm e$, but $\hat{B}$ is not a $U$-nucleus of $G$.
\item $\hat{B}$ is a $U$-nucleus of $G \sm e$, but $W$ is a not a vertex cover of $G$.
\item$\hat{B}$ is a $U$-nucleus of $G \sm e$ and $x,y \not\in W$
\item $\hat{B}$ is a $U/e$-nucleus of $G/e$ and $xy \not\in W'$
\item $\hat B$ is a $U/e$-nucleus of $G/e$, but $\hat B\cup e$ is not a connected subgraph of $G$.
\item $\hat B$ is a $U/e$-nucleus of $G/e$, but $\hat B\cup e$ is not a $U$-nucleus of $G$.
\item $B \in \Delta^{G/e}_{U/e}$ and $B \not\in \Delta^G_U$.
\item $B \in \Delta^{G/e}_{U/e}$ and $\psi^{-1}(B)\not\in \Delta^G_U$.
\end{enumerate}

For $(c) \iff (d)$, regarding a subgraph $\hat{B}$ of $G \setminus e$ as a subgraph of $G$ cannot change connectedness or the vertex set, but may change whether the vertex set is a vertex cover. So if $\hat{B}$ is not a $U$-nucleus of $G$, $W$ must not be a vertex cover of $G$. 

For $(e) \iff (f)$, first note that if $\hat B$ is connected in $G \sm e$ then it is connected in $G/e$. On the other hand, if $x,y \not\in W'$ and $\hat B$ is connected in $G/e$, then $\hat B$ is connected in $G \sm e$. Finally, the property of being a vertex cover is preserved in both directions here.

For $(g) \iff (h)$, if $W'$ is a vertex cover of $G/e$, then $W \cup \{x,y\}=V_G(\hat{B} \cup e)$ must be a vertex cover of $G$. Additionally, $U/e \subseteq W'$ implies $U \subseteq W \cup \{x,y\}$. So if $\hat{B} \cup e$ is not a $U$-nucleus of $G$, it must be disconnected. 
The equivalences $a\iff b\iff c$, $d\iff e$, $f\iff g$, and $h\iff i \iff j$ follow directly from the definitions of nucleus, nucleus complex, and $\psi^{-1}$.
\end{proof}
Now we state and prove the main deletion/contraction recurrence.

\begin{theorem} \label{thm:delete-contract}
Let $G$ be an arbitrary connected graph with $|V(G)|\geq2$.  Let $e\in E(G)$ be neither a loop nor a cut-edge, and let $U \subseteq V(G)$. Then
\[\Tilde{\chi}(\Delta^G_U)=\Tilde{\chi}(\Delta^{G/e}_{U/e})-\Tilde{\chi}(\Delta^{G \sm e}_U).\]
\end{theorem}

\begin{proof}
If $G$ has a loop, then the recurrence is trivially true by Proposition~\ref{elser-complex-loops-and-parallels}~\ref{elser-complex-loops}.

If $G$ has another edge parallel to $e$ (so that contracting $e$ produces a loop), then $\Tilde{\chi}(\Delta^{G/e}_{U/e})=0$ and $\Tilde{\chi}(\Delta^{G}_U)=-\Tilde{\chi}(\Delta^{G \sm e}_U)$ by Proposition~\ref{elser-complex-loops-and-parallels}~\ref{elser-complex-parallels}, implying the recurrence.

If $|V(G)|=2$ and $G$ has no loop, then $G=cK_2$ for some $c\geq1$.  If $c=1$, then no such edge~$e$ exists and the theorem is vacuously true.  If $c>1$, then $G/e$ has a loop, so $\Tilde{\chi}(\Delta^{G/e}_{U/e})=0$ by Proposition~\ref{elser-complex-loops-and-parallels}~\ref{elser-complex-loops}, and the desired recurrence reduces to $\Tilde{\chi}(\Delta^G_U)=-\Tilde{\chi}(\Delta^{G \sm e}_U)=0$, which follows from~\eqref{K2:euler}.

One more case requires special handling.  Suppose that $\Dep(G)=K_3$ and $U=\0$ (so that Proposition~\ref{all-about-psi} does not apply), and that no other edges are parallel to $e$.  Let $a$ and $b$ be the sizes of the other two parallel classes; note that $a,b>0$.  Then $G/e=(a+b)K_2$ and $G-e$ is a graph whose deparallelization is a 3-vertex path.  By Proposition~\ref{elser-complex-loops-and-parallels}~\ref{elser-complex-parallels} together with Proposition~\ref{prop:tree-euler}
and Example~\ref{ex:closed-formulas}~\ref{ex:C3}, we have
\[
\tilde\chi(\Delta^G_\0) = (-1)^{a+b-1},\qquad
\tilde\chi(\Delta^{G/e}_\0) = (-1)^{a+b-1},\qquad
\tilde\chi(\Delta^{G\sm e}_\0) = 0,
\]
so the desired recurrence is satisfied.

In all other cases, the pair $G,U$ satisfies the hypothesis of Proposition~\ref{all-about-psi}, so
\begin{align*}
\Tilde{\chi}(\Delta^G_U) &= \sum_{A \in \Delta^G_U} (-1)^{|A|+1}\\
&=\sum_{\substack{A \in \Delta^G_U:\\e \not\in A}} (-1)^{|A|+1}+\sum_{\substack{A \in \Delta^G_U:\\ e \in A}} (-1)^{|A|+1}\\
&=\sum_{\substack{A \in \Delta^G_U:\\e \not\in A}} (-1)^{|\psi(A)|+1}-\sum_{\substack{A \in \Delta^G_U:\\ e \in A}} (-1)^{|\psi(A)|+1}\\
&=\sum_{\substack{B\in\Delta^{G/e}_{U/e}:\\ \psi^{-1}(B)\in \Delta^G_U}} (-1)^{|B|+1}-\sum_{\substack{B\in\Delta^{G\sm e}_U:\\ \psi^{-1}(B) \in \Delta^G_U}} (-1)^{|B|+1} && \text{(by~\eqref{psi:12})}\\
&=\left(\sum_{B\in\Delta^{G/e}_{U/e}} (-1)^{|B|+1} -
\sum_{\substack{B\in\Delta^{G/e}_{U/e}:\\ \psi^{-1}(B)\not\in \Delta^G_U}} (-1)^{|B|+1}\right)
-\left(\sum_{B\in\Delta^{G\sm e}_U} (-1)^{|B|+1} -
\sum_{\substack{B\in\Delta^{G\sm e}_U:\\ \psi^{-1}(B) \not\in \Delta^G_U}} (-1)^{|B|+1}\right)\\
&=\sum_{B\in\Delta^{G/e}_{U/e}} (-1)^{|B|+1} -\sum_{B\in\Delta^{G\sm e}_U} (-1)^{|B|+1} && \text{(by~\eqref{psi:5})}\\
&=\Tilde{\chi}(\Delta^{G/e}_{U/e})-\Tilde{\chi}(\Delta^{G \sm e}_U).\qedhere
\end{align*}
\end{proof}

\section{Proof of Theorem \ref{thm:els}}
\label{sec:proof-of-conj}

In this section, we combine Theorem~\ref{thm:elser-from-summands}, Proposition~\ref{prop:tree-euler}, and Proposition~\ref{thm:delete-contract} (the deletion-contraction recurrence) to prove Elser's conjecture for all connected graphs.  The idea is to repeatedly apply Proposition~\ref{thm:delete-contract} to edges that are neither loops nor cut-edges, so as to write $\tilde\chi(\Delta^G_U)$ as a signed sum of expressions $\tilde\chi(\Delta^{T_i}_{U_i})$.  Here the graphs $T_i$ are \defterm{tree minors} of $G$, i.e., trees obtained from $G$ by a sequence of deletions and/or contractions.  It will turn out that the signs in this sum are all the same, which will imply immediately that the Elser numbers $\els_k(G)$ are positive for all $k\geq 2$.  This computation can be recorded by a binary tree, which we call a \defterm{restricted deletion/contraction tree}, or RDCT.  To illustrate this idea, we begin with an example.

\begin{example}
Let $G$ be the graph shown below, with the subset $U\subseteq V(G)$ indicated by hollow red circles.
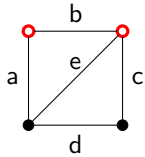
\begin{center}
  \begin{tikzpicture}[scale=1.25]
  \node[point] (A) at (0,0){};
  \node[redpoint] (B) at (0,1){};
  \node[redpoint]  (C) at (1,1){};
  \node[point]  (D) at (1,0){};

  \path[-] (A) edge node[left] {$\mathsf{a}$} (B);
  \path[-] (B) edge node[above] {$\mathsf{b}$} (C);
  \path[-] (A) edge node[below] {$\mathsf{d}$} (D);
  \path[-] (C) edge node[right] {$\mathsf{c}$}(D);
  \path[-] (A) edge node[above] {$\mathsf{e}$} (C);
  \end{tikzpicture}
\end{center}
We can calculate $\tilde\chi(\Delta^G_U)$ by repeated applications of Theorem~\ref{thm:delete-contract}.  One possible set of minors of~$G$ obtained from the recurrence is recorded by the RDCT $\fancyB$ shown in Figure~\ref{impressive-figure}.  The non-leaf nodes of $\fancyB$ are the minors~$H$ with no loops and at least one non-cut-edge $s$; the left and right children are $H\sm s$ and $H/s$ respectively.  The identity of $s$ in each case should be clear from the diagram.  The vertices in $U[H]$ are indicated by hollow red circles; observe that changing the original subset $U\subseteq V(G)$ would change the sets $U[H]$, but not the graphs $H$ themselves.  
The reduced Euler characteristics of the complexes $\Delta^T_U$ are indicated by the numbers at the bottom right of each box. The recurrence stops when it reaches a graph $T$ that is either a tree, in which case Theorem~\ref{thm:delete-contract} does not apply, or has a loop, so that $\tilde\chi(\Delta^T_U)=0$ for all $U$ by Proposition~\ref{elser-complex-loops-and-parallels}(1).  These graphs are precisely the leaves of $\fancyB$.

\begin{figure}[hb]
\begin{center}
\begin{tikzpicture}
\newcommand{\blobsize}{.08}
\newcommand{\vsep}{-2.75}

\newcommand{\hshiftonea}{-4.5}
\newcommand{\hshiftoneb}{4.5}
\newcommand{\hshifttwoa}{-6}
\newcommand{\hshifttwob}{-3}
\newcommand{\hshifttwoc}{3}
\newcommand{\hshifttwod}{6}
\newcommand{\hshiftthreea}{-5}
\newcommand{\hshiftthreeb}{-2}
\newcommand{\hshiftthreec}{2}
\newcommand{\hshiftthreed}{5}
\newcommand{\hshiftfoura}{3.5}
\newcommand{\hshiftfourb}{6.5}+

\begin{scope}[shift={(.5,.5)}]
\draw[thick, dashed]
   (0,0) -- (\hshiftonea,\vsep)
   (0,0) -- (\hshiftoneb,\vsep)
   (\hshiftonea,\vsep) -- (\hshifttwoa,2*\vsep)
   (\hshiftonea,\vsep) -- (\hshifttwob,2*\vsep)
   (\hshiftoneb,\vsep) -- (\hshifttwoc,2*\vsep)
   (\hshiftoneb,\vsep) -- (\hshifttwod,2*\vsep)
   (\hshifttwob,2*\vsep) -- (\hshiftthreea,3*\vsep)
   (\hshifttwob,2*\vsep) -- (\hshiftthreeb,3*\vsep)
   (\hshifttwoc,2*\vsep) -- (\hshiftthreec,3*\vsep)
   (\hshifttwoc,2*\vsep) -- (\hshiftthreed,3*\vsep)
   (\hshiftthreed,3*\vsep) -- (\hshiftfoura,4*\vsep)
   (\hshiftthreed,3*\vsep) -- (\hshiftfourb,4*\vsep);
\end{scope}

\node at (.5,1.25*\vsep) {\Huge$\fancyB$};

\node at (.5,2) {\Large$G$};
\draw[fill=white] (-.6,-.6) rectangle (1.6,1.6); 
\draw[thick] (0,0) -- (0,1) -- (1,1) -- (1,0) -- (0,0) -- (1,1);
\foreach \x in {0,1} \draw[fill=black] (\x,0) circle (\blobsize);
\foreach \x in {0,1} \draw[red, very thick, fill=white] (\x,1) circle (\blobsize);
\node at (-.2,.5) {$\mathsf{a}$};
\node at (.5,1.2) {$\mathsf{b}$};
\node at (1.2,.5) {$\mathsf{c}$};
\node at (.5,-.2) {$\mathsf{d}$};
\node at (.35,.65) {$\mathsf{e}$};
\node at (1.3,-.3) {\Blue{$\mathsf{+1}$}};

\begin{scope}[shift={(\hshiftonea,\vsep)}]
\draw[fill=white] (-.6,-.6) rectangle (1.6,1.6);
\draw[thick] (0,0) -- (0,1) -- (1,1) -- (1,0)  (0,0) -- (1,1);
\foreach \x in {0,1} \draw[fill=black] (\x,0) circle (\blobsize);
\foreach \x in {0,1} \draw[red, very thick, fill=white] (\x,1) circle (\blobsize);
\node at (-.2,.5) {$\mathsf{a}$};
\node at (.5,1.2) {$\mathsf{b}$};
\node at (1.2,.5) {$\mathsf{c}$};
\node at (.35,.65) {$\mathsf{e}$};
\node at (1.3,-.3) {\Blue{$\mathsf{0}$}};
\end{scope}

\begin{scope}[shift={(\hshiftoneb,\vsep)}] 
\draw[fill=white] (-.6,-.6) rectangle (1.6,1.6);
\draw[thick] (0,0) -- (0,1) -- (1,1) -- (0,0);
\draw[thick] (0,0) to[bend right] (1,1);
\draw[fill=black] (0,0) circle (\blobsize);
\foreach \x in {0,1} \draw[red, very thick, fill=white] (\x,1) circle (\blobsize);
\node at (-.2,.5) {$\mathsf{a}$};
\node at (.5,1.2) {$\mathsf{b}$};
\node at (.8,.2) {$\mathsf{c}$};
\node at (.35,.65) {$\mathsf{e}$};
\node at (1.3,-.3) {\Blue{$\mathsf{-1}$}};
\end{scope}

\begin{scope}[shift={(\hshifttwoa,2*\vsep)}] 
\node at (-1,.5) {\Large$T_1$};
\draw[fill=white] (-.6,-.6) rectangle (1.6,1.6);
\draw[thick] (1,0) --  (1,1) -- (0,1) -- (0,0);
\foreach \x in {0,1} \draw[fill=black] (\x,0) circle (\blobsize);
\foreach \x in {0,1} \draw[red, very thick, fill=white] (\x,1) circle (\blobsize);
\node at (-.2,.5) {$\mathsf{a}$};
\node at (.5,1.2) {$\mathsf{b}$};
\node at (1.2,.5) {$\mathsf{c}$};
\node at (1.3,-.3) {\Blue{$\mathsf{0}$}};
\end{scope}

\begin{scope}[shift={(\hshifttwob,2*\vsep)}] 
\draw[fill=white] (-.6,-.6) rectangle (1.6,1.6);
\draw[thick] (1,1) -- (1,0);
\draw[thick] (0,1) to[bend right] (1,1);
\draw[thick] (0,1) to[bend left] (1,1);
\draw[fill=black] (1,0) circle (\blobsize);
\foreach \x in {0,1} \draw[red, very thick, fill=white] (\x,1) circle (\blobsize);
\node at (.5,.65) {$\mathsf{a}$};
\node at (.5,1.35) {$\mathsf{b}$};
\node at (1.2,.5) {$\mathsf{c}$};
\node at (1.3,-.3) {\Blue{$\mathsf{0}$}};
\end{scope}

\begin{scope}[shift={(\hshifttwoc,2*\vsep)}] 
\draw[fill=white] (-.6,-.6) rectangle (1.6,1.6);
\draw[thick] (0,0) -- (0,1) -- (1,1);
\draw[thick] (0,0) to[bend right] (1,1);
\draw[fill=black] (0,0) circle (\blobsize);
\foreach \x in {0,1} \draw[red, very thick, fill=white] (\x,1) circle (\blobsize);
\node at (-.2,.5) {$\mathsf{a}$};
\node at (.5,1.2) {$\mathsf{b}$};
\node at (.8,.2) {$\mathsf{c}$};
\node at (1.3,-.3) {\Blue{$\mathsf{+1}$}};
\end{scope}
\begin{scope}[shift={(\hshifttwod,2*\vsep)}] 
\node at (2,.5) {\Large$T_7$};
\draw[fill=white] (-.6,-.6) rectangle (1.6,1.6);
\draw[thick] (0,1) to[bend right] (1,1);
\draw[thick] (0,1) to[bend left] (1,1);
\draw[thick] (1.2,1) circle (.2);  
\foreach \x in {0,1} \draw[red, very thick, fill=white] (\x,1) circle (\blobsize);
\node at (.5,.65) {$\mathsf{a}$};
\node at (.5,1.35) {$\mathsf{b}$};
\node at (1.2,.65) {$\mathsf{c}$};
\node at (1.3,-.3) {\Blue{$\mathsf{0}$}};
\end{scope}

\begin{scope}[shift={(\hshiftthreea,3*\vsep)}] 
\node at (-1,.5) {\Large$T_2$};
\draw[fill=white] (-.6,-.6) rectangle (1.6,1.6);
\draw[thick] (1,1) -- (1,0);
\draw[thick] (0,1) to[bend right] (1,1);
\foreach \x in {0,1} \draw[red, very thick, fill=white] (\x,1) circle (\blobsize);
\draw[fill=black] (1,0) circle (\blobsize);
\node at (.5,.65) {$\mathsf{a}$};
\node at (1.2,.5) {$\mathsf{c}$};
\node at (1.3,-.3) {\Blue{$\mathsf{0}$}};
\end{scope}
\begin{scope}[shift={(\hshiftthreeb,3*\vsep)}]  
\node at (2,.5) {\Large$T_3$};
\draw[fill=white] (-.6,-.6) rectangle (1.6,1.6);
\draw[thick] (1,1) -- (1,0);
\draw[thick] (.8,1) circle (.2);  
\draw[red, very thick, fill=white] (1,1) circle (\blobsize);
\draw[fill=black] (1,0) circle (\blobsize);
\node at (.4,1) {$\mathsf{a}$};
\node at (1.2,.5) {$\mathsf{c}$};
\node at (1.3,-.3) {\Blue{$\mathsf{0}$}};
\end{scope}
\begin{scope}[shift={(\hshiftthreec,3*\vsep)}] 
\node at (-1,.5) {\Large$T_4$};
\draw[fill=white] (-.6,-.6) rectangle (1.6,1.6);
\draw[thick] (0,0) -- (0,1);
\draw[thick] (0,0) to[bend right] (1,1);
\foreach \x in {0,1} \draw[red, very thick, fill=white] (\x,1) circle (\blobsize);
\draw[fill=black] (0,0) circle (\blobsize);
\node at (-.2,.5) {$\mathsf{a}$};
\node at (.8,.2) {$\mathsf{c}$};
\node at (1.3,-.3) {\Blue{$\mathsf{-1}$}};
\end{scope}

\begin{scope}[shift={(\hshiftthreed,3*\vsep)}] 
\draw[fill=white] (-.6,-.6) rectangle (1.6,1.6);
\draw[thick] (0,0) to[bend right] (1,1);
\draw[thick] (0,0) to[bend left] (1,1);
\draw[red, very thick, fill=white] (1,1) circle (\blobsize);
\draw[fill=black] (0,0) circle (\blobsize);
\node at (.2,.8) {$\mathsf{a}$};
\node at (.8,.2) {$\mathsf{c}$};
\node at (1.3,-.3) {\Blue{$\mathsf{0}$}};
\end{scope}

\begin{scope}[shift={(\hshiftfoura,4*\vsep)}] 
\node at (-1,.5) {\Large$T_5$};
\draw[fill=white] (-.6,-.6) rectangle (1.6,1.6);
\draw[thick] (0,0) to[bend right] (1,1);
\draw[red, very thick, fill=white] (1,1) circle (\blobsize);
\draw[fill=black] (0,0) circle (\blobsize);
\node at (.8,.2) {$\mathsf{c}$};
\node at (1.3,-.3) {\Blue{$\mathsf{0}$}};
\end{scope}
\begin{scope}[shift={(\hshiftfourb,4*\vsep)}] 
\node at (2,.5) {\Large$T_6$};
\draw[fill=white] (-.6,-.6) rectangle (1.6,1.6);
\draw[thick] (.8,1) circle (.2);  
\draw[red, very thick, fill=white] (1,1) circle (\blobsize);
\node at (.4,1) {$\mathsf{c}$};
\node at (1.3,-.3) {\Blue{$\mathsf{0}$}};
\end{scope}

\end{tikzpicture}
\end{center}
\caption{A restricted deletion/contraction tree (RDCT).  Each node is a minor $H$ of $G$.  Vertices in $U[H]$ are indicated by hollow red circles.  The values of $\tilde\chi(\Delta^H_{U[H]})$ are indicated at the bottom right of each box. \label{impressive-figure}}
\end{figure}
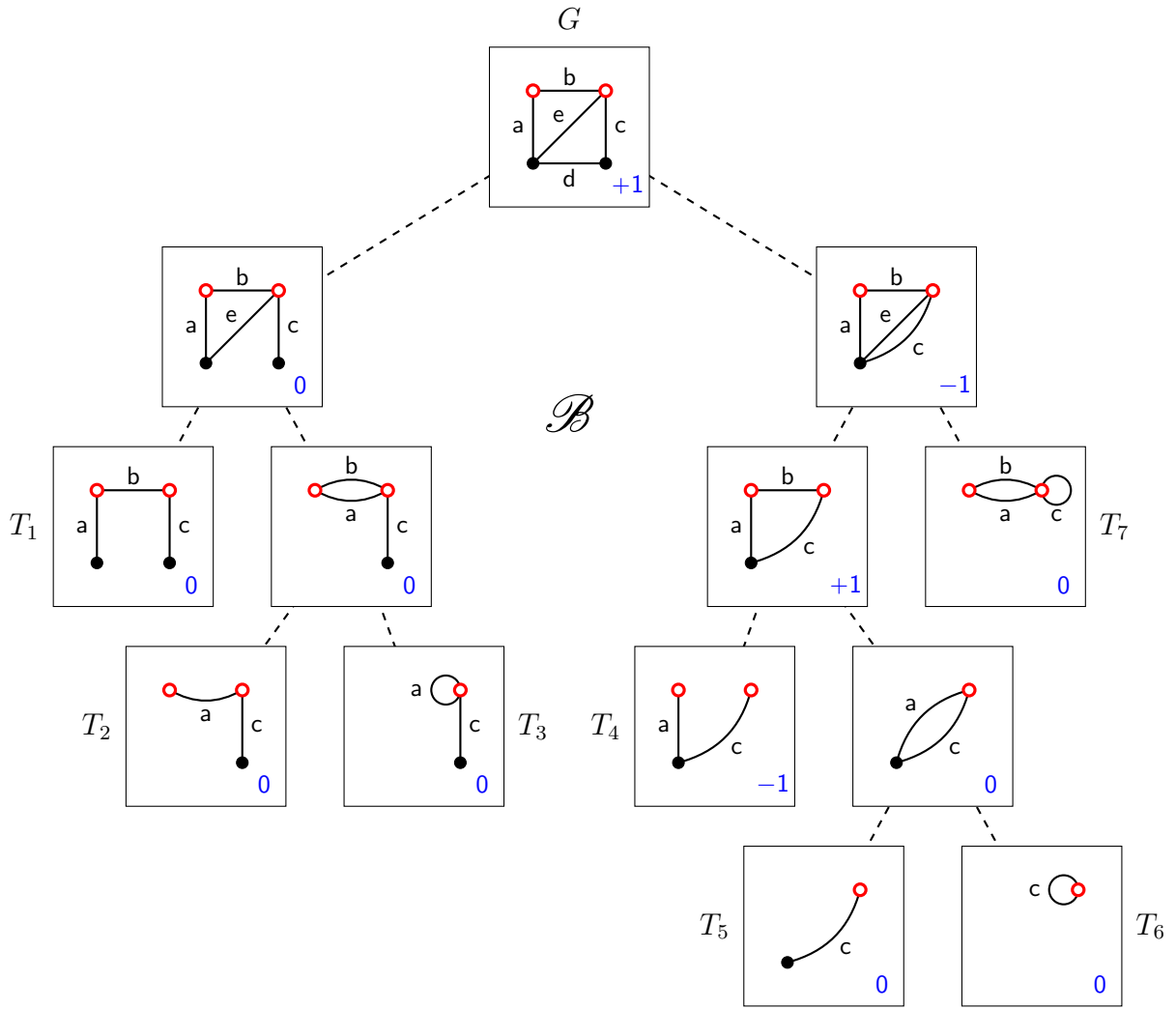
For the tree shown in Figure~\ref{impressive-figure}, we obtain
\begin{equation*}
\Tilde{\chi}(\Delta^G_U)=\sum_{i=1}^7 \Tilde{\chi}(\Delta^{T_i}_{U_i})
\end{equation*}
where the $T_i$ are the leaves of $\fancyB$ and $U_i=U[T_i]$.  The graphs $T_3$, $T_6$, $T_7$ have loops and the summands therefore vanish.  In the other cases, by Theorem \ref{thm:delete-contract}, each deletion changes the sign of the reduced Euler characteristic and each contraction preserves the sign.  So the sign of the $i^{th}$ term is $(-1)^{d_i}$, where $d_i$ is the number of edge deletions required to obtain $T_i$ from $G$. In this case, $d_i=2$ for every $i$, so all the signs are positive; as we will shortly see, this is not an accident.
\end{example}

There are many other possibilities for the RDCT $\fancyB$, depending on which non-cut-edge is chosen at each branch.  Nevertheless, any RDCT for $G$ can be used to compute~$\Tilde{\chi}(\Delta^G_U)$ for any $U$, by giving rise to an equation of the form
\begin{equation} \label{boiled-down-to-trees}
\Tilde{\chi}(\Delta^G_U)=\sum_{i=1}^s \ep_i  ~ \Tilde{\chi}(\Delta^{T_i}_{U[T_i]})
\end{equation}
where $\{T_1,\dots,T_s\}$ is a family of tree minors of $G$, all with at least two vertices and $\ep_i\in\{\pm1\}$ for each~$i$.  The left-hand side is clearly independent of the choice of RDCT; we will see another non-obvious invariant of all RDCTs in Proposition~\ref{thm:els:2}.
In all cases, $|U[T_i]|\leq|U|$, and $U[T_i]=\0$ if and only if $U=\0$.  Moreover, each $T_i$ has at least two vertices because $K_1$ cannot be obtained from a larger simple graph by deleting or contracting a non-cut-edge. Similarly, every maximal sequence consisting of only contractions will eventually contain a loop, when the corresponding summand in~\eqref{boiled-down-to-trees} is 0.

\begin{prop} \label{prop:sign}
In every expression of the form~\eqref{boiled-down-to-trees} arising from an RDCT for~$G$, we have $\ep_i=(-1)^{|E(G)|-|V(G)|+1}$ for all $i\in[s]$.  That is,
\begin{equation} \label{distilled}
\Tilde{\chi}(\Delta^G_U)=(-1)^{|E(G)|-|V(G)|+1}\sum_{i=1}^s \Tilde{\chi}(\Delta^{T_i}_{U[T_i]}).
\end{equation}
\end{prop}

\begin{proof}
Fix $i\in[s]$ and let $T=T_i$.  Obtaining $T$ as a minor of $G$ requires removing a total of $|E(G)|-|E(T)|$ edges via either deletion or contraction.  The number of edges contracted must be $|V(G)|-|V(T)|$, because deletion preserves the number of vertices while contraction reduces it by~1.  Therefore, the number of edges deleted is $(|E(G)|-|E(T)|)-(|V(G)|-|V(T)|)$. Since $T$ is a tree this equals $|E(G)|-|V(G)|+1$ and so $|V(T)|=|E(T)|-1$. The recurrence of Theorem~\ref{thm:delete-contract} implies that the sign $\ep(T)$ is the number of edges deleted.
\end{proof}

\begin{remark}
Proposition~\ref{prop:sign} can also be proved topologically. Consider the homology group $H_1(G;\R)\cong\R^{|E|-|V|+1}$ (where $G$ is regarded as a 1-dimensional cell complex).  Each edge contraction is a homotopy equivalence, hence preserves $H_1$, while each edge deletion lowers the rank of $H_1$ by one.  Since $H_1(T;\R)=0$, the number of deletions must be $|E|-|V|+1$.
\end{remark}

Using these tools, we can now determine the sign of the Elser numbers $\els_k(G)$ for all $G$ and $k$.  The cases $k=0$, $k=1$, and $k\geq 2$ need to be treated separately.

\begin{prop}\label{prop:euler-all}
For any graph $G$ with two or more vertices and any $U\subseteq V(G)$, we have
\begin{enumerate}[label=(\alph*)]
\item If $U=\emptyset$, then \((-1)^{|E(G)|+|V(G)|}\tilde{\chi}(\Delta^G_U) \leq 0\).
\item \label{prop:euler-all:singleton} If $|U|=1$, then \((-1)^{|E(G)|+|V(G)|}\tilde{\chi}(\Delta^G_U) =0\).  (This is \cite[Lemma~1]{Elser}.)
\item If $|U| \geq 1$, then \((-1)^{|E(G)|+|V(G)|}\tilde{\chi}(\Delta^G_U) \geq 0_{}\).
\end{enumerate}
\end{prop}
\begin{proof}
Let $\fancyB$ be an RDCT for~$G$, with leaves labeled by tree minors $T_1,\dots,T_s$. Then~\eqref{distilled} may be rewritten as
\[
(-1)^{|E(G)|-|V(G)|}\tilde{\chi}(\Delta^G_U)=\sum_{i=1}^s -\Tilde{\chi}(\Delta^{T_i}_{U[T_i]}).
\]
By Prop.~\ref{prop:tree-euler}, the summand $-\Tilde{\chi}(\Delta^{T_i}_{U[T_i]})$ equals $-1$ only if $U[T_i]=\0$ (so $U=\0$),
and it equals $+1$ only if $|U[T_i]|\geq2$ (so $|U|\geq2$ as well).
\end{proof}

Now we are ready to prove our main theorem:

\begin{reptheorem}{thm:els}
Let $G$ be a connected graph with at least two vertices, and let $k\geq 0$ be an integer.  Then:
\begin{enumerate}[label=(\alph*)]
\item If $k=0$, then $\els_0(G) \leq 0$.
\item If $k=1$, then $\els_k(G)=0$.  (This is \cite[Theorem~2]{Elser}.)
\item If $k>1$, then $\els_k(G)\geq 0$.  That is, Elser's conjecture holds.
\end{enumerate}
\end{reptheorem}

\begin{proof}
By Theorem~\ref{thm:elser-from-summands}, we have
\begin{equation} \label{denouement-all}
\els_k(G) = \sum_{U\subseteq V(G)} \Sur(k,|U|) \:(-1)^{|E(G)|+|V(G)|} \, \tilde\chi(\Delta^G_U)
\end{equation}
and by Proposition~\ref{prop:euler-all} all summands are nonpositive, zero, or nonnegative according as $k=0$, $k=1$, or $k>1$, implying the result.
\end{proof}

\section{Proof of Theorem~\ref{thm:strict-positivity}}
\label{sec:positivity}

We now consider the question of exactly when the inequalities in (a) and (c) of Theorem~\ref{thm:els} are strict; equivalently, when $\els_k(G)\neq 0$.  We will treat the cases $k=0$ and $k=2$ separately.  Recall from Corollary~\ref{cor:simplify-els} that if $G$ contains a loop, then $\els_k(G)=0$ for all~$k$, and that $\els_k(G)=\els_k(\Dep(G))$.  Therefore, we lose no generality by assuming throughout this section that $G$ is simple.  We begin with the combinatorial interpretation of Elser numbers that can be extracted from the work of the previous section.

\begin{prop} \label{thm:els:2}
Let $G$ be a connected graph with at least two vertices, and let $\fancyB$ be any RDCT for~$G$, whose leaves are tree minors $T_1,\dots,T_s$.  Then:
\begin{enumerate}[label=(\alph*)]
\item\label{els:zero-condition} $\els_0(G)=-\#\{i:\ T_i\cong K_2\}$.
\item\label{els:nonzero-condition} For $k\geq 2$, the following are equivalent:
\begin{itemize}
\item $\els_k(G)>0$;
\item there exists some $U\subseteq V(G)$ such that $|U|\leq k$ and $\tilde\chi(\Delta^G_U)\neq 0$;
\item some tree minor $T_i$ has at most $k$ leaves.
\end{itemize}
\end{enumerate}
\end{prop}

\begin{proof}
Substituting~\eqref{distilled} into~\eqref{denouement-all} gives
\begin{equation} \label{barrel-aged}
\els_k(G) = -\sum_{U\subseteq V(G)} \Sur(k,|U|) \,
\sum_{i=1}^s \Tilde{\chi}(\Delta^{T_i}_{U[T_i]}).
\end{equation}

When $k=0$, equation~\eqref{barrel-aged} simplifies to
\[\els_0(G) = -\sum_{i=1}^s \Tilde{\chi}(\Delta^{T_i}_{\0})\]
which, together with Proposition~\ref{prop:tree-euler}, implies part~\ref{els:zero-condition}.

When $k\geq 2$, all nonzero summands in~\eqref{barrel-aged} must have $1\leq|U|\leq k$ (so that $\Sur(k,|U|)\neq0$) and $U[T_i]\supseteq L_i$ (by Proposition~\ref{prop:tree-euler}), where $L_i$ is the set of leaves of $T_i$.
In particular $|L_i|\leq k$.  On the other hand, if $T_i$ is a tree minor with $\leq k$ leaves occurring as a leaf node of $\fancyB$, then one can pull $L_i$ back under the surjection $V(G)\to V(T_i)$ to obtain a set $U\subseteq V(G)$ with $U[T_i]=L_i$ and $|U|=|U[T_i]|$, so~\eqref{barrel-aged} does indeed have a nonzero summand.
\end{proof}

Proposition~\ref{thm:els:2} is unsatisfactory in that it depends on the choice of a restricted deletion/contraction tree for $G$.  We wish to remove this dependence and give a criterion for nonvanishing that depends only on $G$ itself.  Accordingly, the next goal is to show that \emph{every} tree minor of $G$ appears as a leaf of \emph{some} RDCT.

We begin by recalling some of the theory of 2-connected graphs; see, e.g., \cite[chapter~4.2]{West}.  An \defterm{ear decomposition} of a graph $G$ is a list of subgraphs $R_1,\dots,R_m$ such that
\begin{enumerate}
\item $E(G)=E(R_1)\cup \cdots \cup E(R_m)$;
\item $R_1$ is a cycle; and
\item for each $i>1$, the graph $R_i$ is a path that meets $R_1\cup\cdots\cup R_{i-1}$ only at its endpoints.
\end{enumerate}
It is known that $G$ is 2-connected if and only if it has an ear decomposition (\cite[Thm.~4.2.8]{West}, attributed to Whitney).  Most graphs have many ear decompositions; for instance, $R_0$ can be taken to be any cycle in $G$.  It is easily seen that $m=|E(G)|-|V(G)|+1$, the number of edges in the complement of a spanning tree $T$, suggesting that it ought to be possible to construct an ear decomposition by (essentially) adding a fundamental cycle of $T$ in every iteration.  (A fundamental cycle for $T$ consists of an edge $xy\not\in T$ together with the unique path in $T$ from $x$ to $y$.)

\begin{theorem} \label{all-ears}
Let $G$ be a 2-connected graph and let $T\subseteq G$ be a spanning tree.  Then $G$ has an ear decomposition $R_1\cup\cdots\cup R_m$ such that $|E(R_i)\sm T|=1$ for every $i$.
\end{theorem}
\begin{proof}
We construct the desired ear decomposition by an algorithm that we will first describe informally.  For the cycle $R_1$, we can take any fundamental cycle with respect to $T$ (that is, an edge outside $T$ together with the unique path in $T$ between its endpoints).  At the $i^{th}$ step of the algorithm, we will have constructed a 2-connected graph $G_i=R_1\cup\cdots R_i$ such that $G_i\cap T$ is a spanning tree of $G_i$ (these conditions are loop invariants of the algorithm).  The algorithm then identifies an edge $e\not\in T$ each of whose endpoints can be joined to $G_i$ by (possibly trivial) paths in $T$; these two paths together with $e$ form the ear $R_{i+1}$.

Here is the precise algorithm, including observations that justify its correctness.

\begin{itemize}
  \item \textbf{Initialization:} Let $i=1$, let $R_1$ be any fundamental cycle of $T$, and let $G_1=R_1$.
  \item \textbf{Loop} while $G_i\subsetneq G$:
\begin{itemize}
	\item If there exists an edge $e\in (G\sm T)\sm G_i$ with both endpoints in $V(G_i)$, then let $R_{i+1}=\{e\}$.
	\item Otherwise:
		\begin{itemize}
			\item Let $x$ be a vertex in $V(G_i)$ with at least one neighbor outside $V(G_i)$.
			\item Let $T'$ be the subtree of $T$ consisting of all paths that start at $x$ and take their next step into $V(G_i)$.
			\item Let $T''$ be the subtree of $T$ consisting of all paths that start at $x$ and take their next step outside $V(G_i)$.
			\item Then $E(T)$ is the disjoint union of $E(T')$ and $E(T'')$, and $V(T')\cap V(T'')=\{x\}$.
			\item There must be some edge $e=yz\in E(G)$ with one endpoint $y$ in $V(T')\sm\{x\}$ and one endpoint $z$ in $V(T'')\sm\{x\}$, otherwise $x$ would be a cut-vertex of $G$.
			\item In fact, $e\not\in T$, since $T'$ contains a path $P_y$ from $y$ to $x$ and $T''$ contains a path $P_z$ from $x$ to $z$.
			\item Let $P'$ be the shortest subpath of $P_y$ from $y$ to a vertex in $V(G_i)$, and let $P''=P_z$.
			\item Set $R_{i+1}=P'\cup\{e\}\cup P''$.
		\end{itemize}
  \end{itemize}
	\item In either case, $R_{i+1}$ is a path containing exactly one edge of $T$ and that meets $G_i$ only in its endpoints.
	\item Therefore, the graph $G_{i+1}=G_i\cup R_{i+1}$ is 2-connected.  Moreover, it has $i+1$ ears, each of which contains exactly one edge outside $T$.  Since $T$ is acyclic, it follows that $E(G_{i+1})\cap T$ is a spanning tree.
	\item Increment $i$ and repeat.
\end{itemize}
\end{proof}

\begin{remark}
The algorithm outlined in this proof is essentially equivalent to an algorithm sketched by Fedor Petrov on MathOverflow \cite{Petrov-MO} in response to a question by one of the authors.  Schmidt \cite{Schmidt} proposed a very similar algorithm for determining 2-connectivity and 2-edge connectivity, with the restriction (possibly removable) that $T$ must be a depth-first search tree.
\end{remark}

A \defterm{tree minor} of $G$ is a nontrivial tree of the form $G/C\setminus D$, where $C$ and $D$ are subsets of $E(G)$.  Note that $C$ must be acyclic, and it can be deduced that $|D|=|E(G)|-|V(G)|+1$.  For every RDCT $\fancyB$ of~$G$, every leaf of~$\fancyB$ is a tree minor.  It is not true in general that every tree minor of $G$ actually occurs in some binary tree $\fancyB$, because the order of removing edges has to be arranged to avoid contracting a cut-edge or loop.  For example, if $G$ is the paw graph
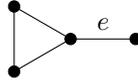
\begin{center}
\begin{tikzpicture}[scale=0.5]
\foreach \a in {0,120,240} \draw[fill=black] (\a:1) circle (.15);
\draw[fill=black] (2.73,0) circle (.15);
\draw (2.73,0) -- (0:1) -- (120:1) -- (240:1) -- (0:1);
\node at (1.87,.4) {$e$};
\end{tikzpicture}
\end{center}
then the tree minor consisting of the cut-edge~$e$ alone cannot appear as a leaf of $\fancyB$.  On the other hand, when $G$ is 2-connected it is possible to achieve every tree minor.

\begin{prop} \label{every-tree-has-its-day}
Let $G$ be 2-connected and let $G/C\backslash D$ be any tree minor, where $C,D\subseteq E(G)$.  Then it is possible to contract the edges of $C$ and delete the edges of $D$ in an order such that one never contracts a cut-edge or deletes a loop.  Therefore, some RDCT of~$G$ contains $G/C\backslash D$ as a leaf.
\end{prop}
\begin{proof}

Let $F=E(G)\sm C\sm D$.  Then $F$ is acyclic; in fact, $T=C\cup F$ must be a spanning tree of $G$
since the tree $G/C\backslash D$ can be produced from $T$ by contracting~$C$, and contraction preserves connectedness and acyclicity.
By Theorem~\ref{all-ears}, there exists an ear decomposition $G=R_1\cup\cdots\cup R_m$ such that $|R_i\sm (C\cup F)|=|R_i\cap D|=1$ for all $i$.

We show by induction on $m$ that it is possible to order the contractions and deletions as desired.  For the base case $m=1$, then $G=R_1$ is an $n$-cycle and $|D|= |E(G)| - |V(G)| + 1 = 1$.  First contract the edges in $C$, of which there can be at most $n-2$, to produce a smaller cycle $G'$, then delete the edge in $D$ which is not a cut edge.

If $m\geq 2$, first contract the edges of $C\cap R_m$.  The result is a (possibly non-simple) graph consisting of $G_{m-1}=R_1\cup\cdots\cup R_{m-1}$ with one additional edge (in $D$)  joining the endpoints of $R_m$.  That edge is not a cut-edge, so we can delete it, leaving the 2-connected graph $G_{m-1}$, and we are done by induction.
\end{proof}

This last observation yields an immediate answer to the question of when $\els_0(G)\neq 0$ (equivalently, by Theorem~\ref{thm:els}, when $\els_0(G)<0$).

\begin{theorem} \label{negativity:els0}
Let $G$ be a connected simple graph.  Then:
\begin{itemize}
\item If $G$ has no cut-vertex, then $\els_0(G)<0$.
\item Otherwise, $\els_0(G)=0$.
\end{itemize}
\end{theorem}
\begin{proof}
The ``otherwise'' case is Theorem~1 of \cite{Elser}.  If $G$ is 2-connected, then let $H$ be any $K_2$ minor.  By Prop~\ref{every-tree-has-its-day}, $H$ appears as a leaf in some RDCT for~$G$, so $\els_0(G)<0$.
\end{proof}

At this point, we have proven Theorem~\ref{thm:strict-positivity} in the cases that $k=0$ and $k=1$.  Accordingly, we assume throughout the rest of the section that $k\geq2$ and $U\neq\0$.

The \defterm{join} of two simplicial complexes $\Gamma_1,\Gamma_2$ on disjoint vertex sets is the simplicial complex $\Gamma_1*\Gamma_2=\{\sigma_1\cup\sigma_2:\ \sigma_1\in\Gamma_1,\;\sigma_2\in\Gamma_2\}$.  A routine calculation shows that $\tilde{\chi}(\Gamma_1 * \Gamma_2)=-\tilde{\chi}(\Gamma_1)\tilde{\chi}(\Gamma_2)$.
In particular, $\tilde{\chi}(\Gamma_1 * \Gamma_2)=0$ if and only if $\tilde{\chi}(\Gamma_1)$ or $\tilde{\chi}(\Gamma_2)$ is zero.

\begin{prop} \label{prop:join}
Let $G$ be a connected graph with a cut-vertex $v$.  Let $G_1,G_2$ be connected subgraphs of $G$, each a union of cut-components of $G$ with respect to $v$, such that
$G_1\cup G_2=G$ and $V(G_1)\cap V(G_2)=\{v\}$.
Let $U \subseteq V(G)$ such that $v\in U$.  Then
\[\Delta_U^G = \Delta_{U \cap V(G_1)}^{G_1} * \Delta_{U \cap V(G_2)}^{G_2}\]
and consequently
\[\tilde{\chi}(\Delta_U^G) = - \tilde{\chi}\left(\Delta_{U \cap V(G_1)}^{G_1}\right) \tilde{\chi}\left(\Delta_{U \cap V(G_2)}^{G_2}\right).\]
\end{prop}

\begin{proof}
Every nucleus of $G$ must contain every cut-vertex by Proposition~\ref{prop:cutsets}, so $\Delta_U^G=\Delta_{U\cup v}^G$ for all $U\subseteq V(G)$, so the assumption $v \in U$ is harmless.  Let $N\in\N(G)$; then $N\cap G_1$ and $N\cap G_2$ are nuclei of $G_1$ and $G_2$, both containing~$v$.  Conversely, if $N_1$ and $N_2$ are nuclei of $G_1$ and  $G_2$ that each contain~$v$, then $N_1\cup N_2$ is a nucleus of $G$ (which of course contains $v$).  Passing to $U$-nucleus complexes by complementing edge sets gives the desired result on joins.  The equation for reduced Euler characteristics follows from the remarks preceding the proposition.
\end{proof}

\begin{prop}\label{not-singleton}
If $G$ is 2-connected and $U \subseteq V(G)$ is nonempty, then $\tilde{\chi}\left(\Delta^G_U\right)=0$ if and only if $|U|=1$.

Consequently, $\els_k(G) > 0$ for all $k \geq 2$.
\end{prop}

\begin{proof}
The case $|U|=1$ is Proposition~\ref{prop:euler-all}\ref{prop:euler-all:singleton}.  Thus, suppose $|U| \geq 2$.
Let $T$ be a spanning tree of $G$.  Let $F$ be the smallest subtree of $T$ such that $V(F)\supseteq U$; in particular $U$ contains all leaves of $F$. Then $F$ is a tree minor of $G$, so by Proposition~\ref{every-tree-has-its-day}, the summand $\tilde\chi(\Delta^F_{U[F]})$ (which equals $-1$ by Proposition~\ref{prop:tree-euler}) appears in some summation expression for $\tilde\chi(\Delta^G_U)$ arising from an RDCT.  Equation~\eqref{distilled} then implies that $\tilde\chi(\Delta^G_U)\neq 0$, and then Proposition~\ref{thm:els:2}\ref{els:nonzero-condition} implies that $\els_k(G)>0$ for all $k \geq 2$.
\end{proof}

Proposition~\ref{not-singleton} completes the proof of Theorem~\ref{thm:strict-positivity}\ref{strict-pos:2conn}.

\begin{prop}\label{block-char}
Let $U \subseteq V(G)$ be nonempty, let $K$ be the collection of cut-vertices of $G$, let $U'=U\cup K$, and let $B_1,\dots B_c$ be the $2$-connected components of $G$.

Then $\tilde{\chi}(\Delta^G_U)=0$ if and only if $|V(B_i) \cap U'| \geq 2$ for every $i$.
\end{prop}
\begin{proof}
Repeatedly applying Prop.~\ref{prop:join} gives
\[
\tilde{\chi}(\Delta_{U'}^G)=(-1)^{c-1}\prod_{i=1}^c \tilde{\chi} \left(\Delta_{{U'} \cap V(B_i)}^{B_i}\right)
\]
which, by Proposition~\ref{not-singleton}, is zero if and only if $|{U'} \cap V(B_i)|=1$ for some $i$.
\end{proof}

Combining Proposition~\ref{thm:els:2}\ref{els:nonzero-condition} with Proposition~\ref{block-char} implies the characterization of the positivity of $\els_k(G)$ in Theorem~\ref{thm:strict-positivity}\ref{strict-pos:not2conn}, completing the proof.

We conclude this section by mentioning an interesting problem, due to an anonymous referee.  Let $f(x)$ be a function, and define the \defterm{generalized Elser invariant} of a graph $G$ by
\[\els_f(G) = (-1)^{|V(G)|+1} \sum_{N\in\N(G)} (-1)^{|E(N)|} f(|V(N)|)\]
so that setting $f(x)=x^k$ gives $\els_k(G)$.  For which functions $f(x)$ does our argument establish the sign of $\els_f(G)$?  Our proof depends only on the signs of the coefficients $\Sur(k,|U|)$, so a sufficient condition is that $\els_f$ satisfies some analogue of Theorem~\ref{thm:elser-from-summands} with appropriate signs.  I.e., suppose that the generalized Elser invariant can be rewritten as
\[
\els_f(G) = (-1)^{|E(G)|+|V(G)|} \sum_{U\subseteq V(G)} a(U) \;\tilde\chi(\Delta^G_U)
\]
such that, for some nonnegative integer $k$, the number $a(U)$ has the same sign as $\Sur(k,|U|)$ for all $U$.  Then our argument implies that $\els_f(G)$ has the same sign as $\els_k(G)$.  It would be interesting to look for other graph invariants with this property.

\section{Monotonicity} \label{sec:monotonicity}

In this section, we use the technical results of Sections~\ref{section:deletion-contraction} and~\ref{sec:proof-of-conj}, including the proof of Elser's conjecture itself, to prove a deletion-contraction type \emph{inequality} for Elser numbers that is stronger than the original conjecture.

\begin{theorem} \label{thm:monotonic}
Let $e\in E(G)$ such that $e$ is not a loop or cut-edge. Then
\[\els_k(G) \geq \els_k(G/e)+\els_k(G\sm e),\]
with equality for $k=0$.
\end{theorem}

\begin{proof}
Let $x$ and $y$ be the endpoints of $e$.
Abbreviating $n=|V(G)|=|V(G\sm e)| = |V(G/e)| +1$ and $m=|E(G)|=|E(G\sm e)|+1 = |E(G/e)| +1$, we have
\begin{align*}
\els_k(G)-\els_k(G\sm e)
&= (-1)^{m+n} \sum_{U\subseteq V(G)} \Sur(k,|U|) \;\tilde\chi(\Delta^G_U)-(-1)^{m+n-1} \sum_{U\subseteq V(G\sm e)} \Sur(k,|U|) \;\tilde\chi(\Delta^{G\sm e}_U)\\
\intertext{(by Theorem~\ref{thm:elser-from-summands})}
&= (-1)^{m+n} \sum_{U\subseteq V(G)} \Sur(k,|U|) \left( \tilde\chi(\Delta^G_U)+\tilde\chi(\Delta^{G\sm e}_U)\right)\\
&= (-1)^{m+n} \sum_{U\subseteq V(G)} \Sur(k,|U|) \;\tilde\chi(\Delta^{G/e}_{U/e})\\
\intertext{(by Theorem~\ref{thm:delete-contract})}
&= (-1)^{|E(G/e)|+|V(G/e)|} \left( \sum_{\substack{U\subseteq V(G)\\ x\not\in U}} \Sur(k,|U|) \;\tilde\chi(\Delta^{G/e}_{U/e}) + \sum_{\substack{U\subseteq V(G)\\ x\in U}} \Sur(k,|U|) \;\tilde\chi(\Delta^{G/e}_{U/e})\right)\\
&= \els_k(G/e)+(-1)^{|E(G/e)|+|V(G/e)|} \sum_{\{x\}\subseteq U\subseteq V(G)} \Sur(k,|U|) \;\tilde\chi(\Delta^{G/e}_{U/e})\\
\intertext{(rewriting the first sum via the natural bijection $V(G)\sm\{x\}\to V(G/e)$ sending $y$ to $xy$)}
&\geq \els_k(G/e)
\end{align*}
(by Theorem \ref{thm:els}).  Note that when $k=0$, the last sum vanishes and the last inequality is an equality.
\end{proof}

As in Section \ref{sec:proof-of-conj}, we can iterate this recurrence until we obtain a tree. So for any graph $G$, we have that $\els_k(G)$ is bounded below by $\sum_{i}\els_k(T_i)$ for a collection of trees $T_i$ when $k>1$. We illustrate this in the following simple example.

\begin{example}
Let $n\geq 3$ and consider $C_{n}$, the cycle graph on $n$ vertices. Let $e\in E(C_{n})$ be any edge of $C_{n}$. Then $C_{n}\backslash e$ is $P_{n}$ and $C_{n}/e$ is $C_{n-1}$.
By Theorem \ref{thm:monotonic}, we have
\begin{align*}
\els_k(C_{n}) & \geq \els_k(C_{n-1})+\els_k(P_{n}).
\end{align*}
Iterating this gives
\begin{align*}
\els_k(C_{n}) & \geq \sum_{i=1}^n \els_k(P_{i})
\geq \sum_{i=1}^n (i+1)^k - 2i^k + (i-1)^k
\end{align*}
by Example \ref{ex:closed-formulas}.
\end{example}

\section{Nucleus complexes: future directions} \label{sec:future}

In this last section, we explore combinatorial and topological aspects of nucleus complexes, in many cases without giving proofs.
We had initially intended to prove Elser's conjecture by computing their simplicial homology groups and thus their Euler characteristics.  While this approach did not prove feasible, nucleus complexes nonetheless appear to be interesting objects in their own right, worthy of future study.

We begin with some easy observations.  Let $G$ be a connected graph and $U,U'\subseteq V(G)$.  It follows easily from the definition of nucleus complexes that if $U \subseteq U'$, then $\Delta^G_{U'}\subseteq\Delta^G_U$.  Moreover, in all cases, ${\Delta^G_U\cap\Delta^G_{U'} = \Delta^G_{U\cup U'}}$.  On the other hand, ${\Delta^G_U\cup\Delta^G_{U'} \subseteq \Delta^G_{U\cap U'}}$, but equality need not hold.

A \defterm{matroid on ground set $E$} (more properly, a matroid independence complex) is a simplicial complex $M$ on vertices $E$ with the property that if $\sigma,\tau\in M$ and $|\sigma|>|\tau|$, then there is a vertex $v\in\sigma\sm\tau$ such that $\tau\cup\{v\}\in M$.  For a general reference on matroids, see, for example, \cite{Oxley}; for matroid complexes, see \cite{CCA}.  Every connected graph $G$ has an associated \defterm{graphic matroid} $M(G)=\{A\subseteq E(G):\ A\text{ is acyclic}\}$ and \defterm{cographic matroid} $M^*(G)=\{A\subseteq E(G):\ G\sm A\text{ is connected}\}$, of dimensions $|V(G)|-2$ and $|E(G)|-|V(G)|$ respectively.  These matroids are dual; that is, the facets (maximal faces) of $M^*(G)$ are precisely the complements of facets of $M(G)$.

In fact, $\Delta^G_{V(G)}$ is precisely the cographic matroid $M^*(G)$.  In particular, it is shellable, homotopy-equivalent to a wedge of spheres of dimension $|E(G)|-|V(G)|$, and has homology concentrated in that dimension.  For arbitrary $U\subseteq V(G)$, the nucleus complex $\Delta^G_U$ is not in general a matroid complex.  Nevertheless, experimental data gathered using Sage~\cite{Sage} supports the following conjecture.

\begin{conj} \label{homology-concentration}
Let $G$ be a connected graph and $U\subseteq V(G)$.  Then the reduced homology group $\HH_k(\Delta^G_U;\R)$ is nonzero only if (i) $U=\0$ and $k=|E(G)|-|V(G)|+1$, or (ii) $|U|\geq2$ and $k=|E(G)|-|V(G)|$.
\end{conj}

By Proposition~\ref{prop:join}, it is enough to prove the conjecture in the case that $G$ is 2-connected.  Using Sage, we have verified the conjecture computationally for all 2-connected graphs with 6 or fewer vertices

\begin{prob}
Compute the Betti numbers $\dim\HH_k(\Delta^G_U;\R)$ combinatorially for arbitrary $G,U,k$.
\end{prob}

A partial proof of Conjecture~\ref{homology-concentration} can be obtained using Jonsson's theory of \emph{pseudo-independence complexes}; we refer the reader to \cite[chapter~13]{Jonsson} for the relevant definitions and theorems.  In short, it can be shown that in all cases, the nucleus complex $\Delta^G_U$ is a \emph{pseudo-independence (PI) complex} over the matroid $M^*(G)$ for all $U$, and a \emph{strong pseudo-independence (SPI) complex} whenever $U$ is a vertex cover.  It follows that $\HH_k(\Delta^G_U;\R)=0$ for all $k<|E(G)|-|V(G)|$ and all $U\subseteq V(G)$, and for all $k\neq|E(G)|-|V(G)|$ when $U$ is a vertex cover.  However, if $U$ is not a vertex cover, then $\Delta^G_U$ sometimes fails to be SPI over $M^*(G)$.

In another direction, one can ask how Elser numbers depend on the graphic matroid $M(G)$.  Interestingly, while $\els_k(G)$ cannot be a matroid invariant for $k>1$ (since it is not constant on trees with the same number of edges, all of which have isomorphic graphic matroids), it turns out that $\els_0(G)$ is a matroid invariant for 2-connected graphs.  This fact can be proven using Whitney's characterization of graphic matroid isomorphism in terms of 2-switches \cite{Whitney}.  Even in light of the $k=0$ case of Theorem~\ref{thm:monotonic}, it is not clear whether $\els_0(G)$ can be obtained from the Tutte polynomial: it is negative on 2-connected graphs but zero on graphs with a cut-vertex (cf.\ Theorem~\ref{thm:strict-positivity}\ref{strict-pos:not2conn}), hence not multiplicative on direct sums.  For $k\geq 2$, $\els_k(G)$ is not a matroid invariant even for 2-connected graphs; for example, the 2-connected graphs $G_1$ and $G_2$ shown below have isomorphic graphic matroids, but $\els_2(G_1)=42$ and $\els_2(G_2)=44$.
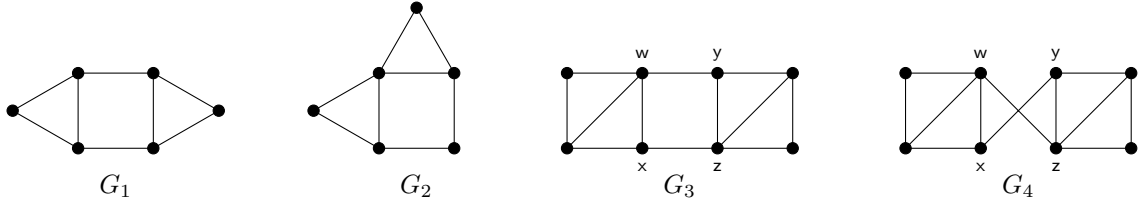
\begin{center}
\begin{tikzpicture}
\newcommand{\vxsize}{.075}
\draw (0,0) -- (0,1) -- (1,1) -- (1,0) -- cycle  (0,0) -- (-.87,.5) -- (0,1)  (1,0) -- (1.87,.5) -- (1,1);
\foreach \x/\y in {0/0, 1/0, 0/1, 1/1, -.87/.5, 1.87/.5} \draw [fill=black] (\x,\y) circle(\vxsize);
\node at (.5,-.5) {$G_1$};
\begin{scope}[shift={(4,0)}]
\draw (0,0) -- (0,1) -- (1,1) -- (1,0) -- cycle  (0,0) -- (-.87,.5) -- (0,1) -- (.5,1.87) -- (1,1);
\foreach \x/\y in {0/0, 1/0, 0/1, 1/1, -.87/.5, .5/1.87} \draw [fill=black] (\x,\y) circle(\vxsize);
\node at (.5,-.5) {$G_2$};
\end{scope}
\begin{scope}[shift={(6.5,0)}]
\draw (0,0) -- (3,0) -- (3,1) -- (0,1) -- cycle (1,0)--(1,1) (2,0)--(2,1) (0,0)--(1,1) (2,0)--(3,1);
\foreach \x in {0,1,2,3} { \foreach \y in {0,1} \draw [fill=black] (\x,\y) circle(\vxsize); }
\node at (1.5,-.5) {$G_3$};
\node at (1,1.25) {\scriptsize$\mathsf{w}$};
\node at (1,-.25) {\scriptsize$\mathsf{x}$};
\node at (2,1.25) {\scriptsize$\mathsf{y}$};
\node at (2,-.25) {\scriptsize$\mathsf{z}$};
\end{scope}
\begin{scope}[shift={(11,0)}]
\draw (0,0) -- (1,0) -- (1,1) -- (0,1) -- cycle (2,0) -- (2,1) -- (3,1) -- (3,0) -- cycle (1,0)--(2,1) (1,1)--(2,0) (0,0)--(1,1) (2,0)--(3,1);
\foreach \x in {0,1,2,3} { \foreach \y in {0,1} \draw [fill=black] (\x,\y) circle(\vxsize); }
\node at (1.5,-.5) {$G_4$};
\node at (1,1.25) {\scriptsize$\mathsf{w}$};
\node at (1,-.25) {\scriptsize$\mathsf{x}$};
\node at (2,1.25) {\scriptsize$\mathsf{y}$};
\node at (2,-.25) {\scriptsize$\mathsf{z}$};
\end{scope}
\end{tikzpicture}
\end{center}
On the other hand, $M(G_3)\cong M(G_4)$, and $\els_k(G_3)=\els_k(G_4)$ for all $k\geq2$.  In general, if $G'$ is obtained from $G$ by replacing an edge cut $\{wy,xz\}$ with another edge cut $\{wz,xy\}$ (as for the pair $G_3,G_4$ above), then there is a bijection $\N(G)\to\N(G')$ that preserves vertex sets and edge set cardinalities, so $\tilde\chi(\Delta^G_U)=\tilde\chi(\Delta^{G'}_U)$ and $\els_k(G)=\els_k(G')$ for all $U$ and $k$.  It is possible that there are other special 2-switches with the same properties.

\section*{Acknowledgements}
 The authors thank Veit Elser for proposing the problem, Lou Billera for bringing it to our attention, and Vic Reiner for suggesting the topological approach. We thank the Graduate Research Workshop in Combinatorics for providing the platform for this collaboration as well as the Institute for Mathematics and its Applications and the Center for Graduate and Professional Diversity Initiatives at the University of Kentucky for further financial support.  In addition, we thank two anonymous referees for their careful reading and helpful suggestions.

\raggedright
\bibliographystyle{amsalpha}
\bibliography{biblio}
\end{document}